\documentclass[11pt]{amsart}

\usepackage{graphicx}
\usepackage{mathrsfs}
\usepackage{oubraces}
\usepackage{float}
\usepackage{array,booktabs}
\usepackage{enumerate}
\usepackage{amsfonts}
\usepackage{amsmath}
\usepackage{amssymb}
\usepackage{tabularx}
\usepackage{blkarray}
\usepackage{multirow}
\usepackage{bigdelim}

\usepackage{hyperref}

\usepackage{latexsym}

\newcommand\s{\sigma}%\s == \sigma

\newcommand\w{\omega}

\newcommand\Lam{\Lambda}
\newcommand\ep{\varepsilon}%\ep == \varepsilon
\newcommand\Vp{\varpi}%\vp == \varphi 
\newcommand\vp{\varphi}%\vp == \varphi 
\newcommand\oo{\mathfrak{o}}
\newcommand\Sf{{\mathfrak{S}}}
\newcommand\Af{{\mathfrak{A}}}

\newcommand\Bf{\mathfrak{B}}
\newcommand\Cf{\mathfrak{C}}
\newcommand\Df{\mathfrak{D}}
\newcommand\If{\mathfrak{I}}
\newcommand\Jf{\mathfrak{J}}

\newcommand\p{{\mathfrak{p}}}
\newcommand\cf{{\mathfrak{c}}}
\newcommand\gf{\mathfrak{g}}
\newcommand\dr{\mathrm{d}}

\newcommand\xb{{\bold{x}}}
\newcommand\yb{{\bold{y}}}
\newcommand\zb{{\bold{z}}}

\newcommand\C{{\mathbb{C}}}%\C == \mathbb{C}

\newcommand\Z{{\mathbb{Z}}}

\newcommand\K{{\mathbb{K}}}

\newcommand\Sb{{\mathbb{S}}}
\newcommand\Ib{{\mathbb{I}}}
\newcommand\Wb{{\mathbb{W}}}

\newcommand\Ss{{\mathscr{S}}}

\newcommand\Nc{{\mathcal{N}}}

\newcommand\Rc{{\mathcal{R}}}

\newcommand\Hc{{\mathcal{H}}}
\newcommand\Oc{{\mathcal{O}}}

\newcommand\1{\mathbf{1}}

\newcommand\diag{{\rm diag}}
\newcommand\tr{\mathrm{tr}}
\newcommand\nid{\noindent}
\newcommand\ac{\acute}
\newcommand\ol{\overline}%\ol == \overline
\newcommand\lt{\ltimes}%\lt == \ltimes
\newcommand\bs{\backslash}%\bs == \backslash
\newcommand\GL{{\mathrm{GL}}}

\newcommand\SL{{\mathrm{SL}}}
\newcommand\GSp{{\mathrm{GSp}}}

\newcommand\In{{\mathrm{Ind}}}%\In == {\rm Ind}
\newcommand\Hom{{\mathrm{Hom}}}%\Hom == {\rm Hom}

\newcommand\cI{{\rm c\mathchar`- Ind}}
\newcommand\tm{\times}

\newcommand\Ch{{\mathrm{Ch}}}%\Ch == {\rm Ch}
\newcommand\supp{{\mathrm{supp}}}%\spo == {\rm supp}
\newcommand\vol{{\mathrm{vol}}}

\newcommand\Ir{{\mathscr{I}}}
\newcommand\Alg{{\mathscr{A}}}

\newcommand\bu{\bar{u}}

\numberwithin{equation}{section}

\newtheorem{thm}{Theorem}[section]

\newtheorem{prop}[thm]{Proposition}
\newtheorem{Cor}[thm]{Corollary}
\newtheorem{lem}[thm]{Lemma}
\newtheorem{rem}[thm]{Remark}

\address{Department of Mathematical and Physical Sciences, Nara Woman University,
Kitauoyahigashi-machi, Nara 630-8506, Japan. }
\email{okazaki@cc.nara-wu.ac.jp}
\begin{document}

\title{Shalika newforms for GL($n$)}

\dedicatory{}

\author{Takeo Okazaki}


\maketitle

\begin{abstract}
Let $(\pi,V)$ be a generic irreducible representation of a general linear group over a $p$-adic field.
Jacquet, Piatetski-Shapiro, and Shalika gave an open compact subgroup $K$, so that the subspace $V^{K}$ consisting of $v \in V$ fixed by $K$ is one-dimensional.
If $\pi$ has a Shalika model $\Lambda$, then we call vectors in $\Lambda(V)$ the Shalika forms of $\pi$, and those in $\Lambda(V^{K})$ the Shalika newforms.
In this article, we give a method to determine all values of the Shalika newforms on the mirabolic subgroup in the case where $\pi$ is supercuspidal.
Using this result, we give another Shalika form with nice properties, which is not fixed by $K$ in the case where the character defining the Shalika model is ramified.
\end{abstract}

\tableofcontents

\section{Main results}\label{sec:main}
Let $n = 2m$ be an even integer.
Let $F$ be a $p$-adic field, and $S$ the Shalika subgroup of $G_{n} = \GL_{n}(F)$ consisting of matrices 
\begin{align*}
s = \begin{bmatrix}
a &b \\
&a
\end{bmatrix}.
\end{align*} 
Let $\psi: F \to \C^\tm$ be a nontrivial additive character.
Let $\chi: F^\tm \to \C^\tm$ be a continuous homomorphism.
If an irreducible representation $\pi$ of $G_{n}$ is realized in a subspace $\Sb_\pi(\chi)$ of the space consisting of continuous functions $J: G_{n} \to \C$ such that 
\begin{align*}
J(sg) = \chi \circ \det(a) \psi(\tr(a^{-1}b))J(g), 
\end{align*} 
then we call $J \in \Sb_\pi(\chi)$ a Shalika form of $\pi$ relevant to $\chi$.
Suppose that $\pi$ is generic with conductor $\cf_\pi$ and central character $\w_\pi$.
By the work of Jacquet, Piatetski-Shapiro, and Shalika \cite{JPSS}, there exists a unique $J \in \Sb_\pi(\chi)$ up to constant multiples such that 
\begin{align*}
\pi(k)J = \w_\pi(k_{n,n})J
\end{align*} 
for $k$ lying in the open compact subgroup
\begin{align*}
\Gamma_n(\cf_\pi) & = \Gamma(\cf_\pi) := \GL_{n}(\oo) \cap \begin{bmatrix}
1_{n-1}&  \\
& \Vp^{\cf_\pi}
\end{bmatrix}\GL_{n}(\oo)\begin{bmatrix}
1_{n-1}&  \\
& \Vp^{\cf_\pi}
\end{bmatrix}^{-1} \\
&= \{k \in \GL_{n}(\oo) \mid k_{n,1}, \ldots,  k_{n,n-1} \in \p^{\cf_\pi} \} 
\end{align*}  
where $\oo$ indicates the ring of integers of $F$, $\p = \Vp \oo$ the prime ideal, and $1_{n-1}$ the unit matrix of $G_{n-1}$.
We call such a Shalika form, a Shalika newform of $\pi$, and denote by $J^{new}$.

In this article, for supercuspidal representations, we give a method to determine all values of $J^{new}$ on the mirabolic subgroup $P_{n}$.
By it, we obtain: 
\begin{thm}\label{thm:main}
Let $\pi$ be a generic, irreducible, supercuspidal representation of $G_{n}$ realized in $\Sb_\pi(\chi)$.
Let $e$ be the conductor of $\chi$.
Let $P_n(\oo) \subset P_n$ denote the subgroup consisting matrices $p \in P_n$ with entries in $\oo$.
Assume that $\psi(\oo) = \{1\} \neq \psi(\p^{-1})$.
Then we have the followings.
\begin{enumerate}[i)]
\item If $e=0$, then the support of $J^{new}|_{P_{n}}$ contains $(S \cap P_{n}) P_{n}(\oo)$.
\item If $e > 0$, then the support of $J^{new}|_{P_{n}}$ equals the $(S \cap P_{n}, P_{n}(\oo))$-double coset of
\begin{align*}
g_n:= 
\left[
\begin{array}{ccccc|c}
\Vp^{e} &  &  &  & 1 &  \\ 
 & \Vp^{3e} &  & &\Vp^{e} &  \\ 
&  &\ddots & & \vdots & \\
 & & & \Vp^{(n-3)e} & \Vp^{(m-2)e} &  \\
 & & & & \Vp^{(m-1)e}&  \\ \hline
 &  & & & & 1_{m}
\end{array}
\right].
\end{align*} 
\end{enumerate}
\end{thm}
\nid
Also for non-supercuspidal generic representations whose standard $L$-function equal $1$, the same thing holds under the assumption (\ref{ass:J1nz}) (c.f. Proposition \ref{prop:asaV}, Theorem. \ref{thm:suppJ}, \ref{thm:unrpsi}, \ref{thm:Ln}).

However, from our aesthetic perspective and various application ones, we think it is not desirable that $J^{new}$ vanishes at $1_n$ in the case where $\chi$ is ramified, contrary to the fact that Whittaker newforms never vanish at $1_n$.
Further, from some results of representations $(\tau,V)$ of algebraic groups, in the case where an $L$-function $L(s,\tau)$ is defined as a generator of a fractional ideal of the principal ideal domain $\C[q^{s},q^{-s}]$ ($q$ is the cardinality of $\oo/\p$) spanned by some zeta integrals, and a functional equation between $L(s,\tau)$ and $L(1-s,\tau^\vee)$ for the contragredient $\tau^\vee$ holds, we infer that the desirable newform theory should satisfy the followings.
\begin{enumerate}[i)]
\item for a suitable open compact subgroup $K$, and a homomorphism $\Omega: K \to \C^\tm$, the subspace $V^K_\Omega := \{v \in V \mid \tau(k)v = \Omega(k) v \}$ is one-dimensional.
\item the open compact subgroup $K$ depends on and the realization $V$, and the definition of the zeta integral.
\item if $V$ consists of functions on $G$, then $\xi \in V^K_\Omega$ does not vanish at the identity.
\item the zeta integral or its `suitable' arrangement of $\xi \in V^K_\Omega$ coincides with $L(s,\tau)$, and that of its suitable conjugate of $\xi$ coincides with the product of $L(s,\tau^\vee)$ and the root number appearing in the functional equation.
\end{enumerate}
By this reason, we construct a Shalika form $J_\pi$ with the following properties, assuming: 
\begin{itemize}
\item $\pi$ is unitary.
\item $J^{new}$ does not vanish at $1_n$ (resp. the above $g_n$) in the case of $e = 0$ (resp. $e>0$).
\item $\cf_\pi \ge me$. 
\end{itemize}
Put
\begin{align*}
l = \cf_\pi- (m-1)e. 
\end{align*}
Let $\Oc_r$ denote the ring of $r \tm r$ matrices with entries in $\oo$, and for $s \in \Z$, 
\begin{align*}
\Oc_r(s) &= \begin{bmatrix}
1_{r-1}&  \\
& \Vp^s
\end{bmatrix} \Oc_r \begin{bmatrix}
1_{r-1}&  \\
& \Vp^s
\end{bmatrix}^{-1}\\
\mathcal{R}_r(s) &= \Oc_r \cap \Oc_r(s).
\end{align*} 
We define the ring 
\begin{align*}
R_{\cf_\pi} = \Rc_n(l) \cap \begin{bmatrix}
1_{m}&  \\
& \Vp^{e}1_m
\end{bmatrix}\Oc_{n}\begin{bmatrix}
1_{m}&  \\
& \Vp^{e}1_m
\end{bmatrix}^{-1},
\end{align*} 
and denote by $\K(\cf_\pi)$ the units group of $R_{\cf_\pi}$, which is an open compact subgroup of $G_n$ and equals $\Gamma(\cf_\pi)$ in the case where $\chi$ is unramified.
Explicitly, $R_{\cf_\pi} $ consists of matrices 
\begin{align}
\begin{bmatrix}
A & B  \\
C& D
\end{bmatrix}, \ D \in \Rc_m(l),\ A, B, \begin{bmatrix}
\Vp^e 1_{m-1}&  \\
& \Vp^l
\end{bmatrix}^{-1} C \in \Oc_m. \label{eq:Rc}
\end{align} 
The first property of $J_\pi$ is that
\begin{align*}
\pi(k)J_\pi = \chi \circ \det(d_k) J_\pi, \ k \in \K(\cf_\pi)
\end{align*}
where $d_k$ indicates the $m \tm m$ block matrix of $k$ in the lower right corner.
For $i \in \Z$, let 
\begin{align*}
B_{m,i}= \{b \in B_m \cap \Oc_m \mid \det(b) \in \Vp^i \oo^\tm \},
\end{align*} 
where $B_m$ indicates the Borel subgroup of $G_m$.
The second property of $J_\pi$ is 
\begin{align*}
L(s,\pi) = \sum_{i = 0}^\infty c_i q^{i(-s+\frac{1}{2})}, c_i = \sum_{b \in B_{m,i}/B_{m,0}} J_\pi(\begin{bmatrix}
b &  \\
& 1_m
\end{bmatrix}).
\end{align*} 
In particular, $J_\pi(1_n) = 1$.
Set
\begin{align}
\upsilon_{\cf_\pi} = \begin{bmatrix}
\Vp^l& &  \\
& \Vp^e 1_{m-1}&  \\
& & w_m
\end{bmatrix} \in G_n,
\label{def:ups}
\end{align} 
where $w_m$ indicates the standard unti-diagonal Weyl element in $G_m$.
Define $J_\pi^* \in \Sb_{\pi^\vee}(\chi^{-1})$ by
\begin{align*}
J_\pi^*(g) = J_\pi(w_{n} {}^tg^{-1} \upsilon_{\cf_\pi}).
\end{align*} 
Let $\K(\cf_\pi)^* = \upsilon_{\cf_\pi}^{-1} \K(\cf_\pi) \upsilon_{\cf_\pi}$, which is the units group of the ring $R_{\cf_\pi}^*$ consisting of matrices 
\begin{align}
\begin{split}
& \begin{bmatrix}
A & B  \\
C& D
\end{bmatrix},  D \in w_m{}^t\Rc_m(l)w_m, \\ & A, B \in w_m{}^t\Oc_m(l')w_m,  C \in \Oc_m\begin{bmatrix}
\Vp^l &  \\
& \Vp^e 1_{m-1}
\end{bmatrix}
\end{split}\label{eq:Rc*}
\end{align} 
where $l' = l-e$. 
From the first property of $J_\pi$, it follows that 
\begin{align*} 
\pi(k)J_\pi^* = \chi \circ \det(d_k)^{-1} J_\pi^*, \ k \in \K(\cf_\pi)^*.
\end{align*}
For $k \in \Z$, let 
\begin{align*}
B_{m,k}^e = \{b \in B_m \mid \det(b) \in \Vp^k \oo^\tm, b_{11} \in \oo, b_{1j} \in \p^{-l'}; j >1 , b_{ij} \in \oo; i>1 \}.
\end{align*} 
Then, 
\begin{align*}
\ep_\pi L(s,\pi^\vee) = \sum_{i = 0}^\infty c_i^* q^{i(-s+\frac{1}{2})}, c_i^* = \sum_{b \in B_{m,k}^e/B_{m,0}^e} J_\pi^*(\begin{bmatrix}
b &  \\
& 1_m
\end{bmatrix}).
\end{align*} 
In particular, $J_\pi^*(1_n)$ equals $\ep_\pi$, the root number of $\pi$.
So, our $J_\pi$ and $\K(\cf_\pi)$ satisfies the above conditions except for i), the one-dimensionality.
This problem will be discussed in a future work.

Another motivation of the above construction is a theta lift from $G_4$ to $\GSp_4(F)$.
It is known that each generic, irreducible, admissible representation of $\GSp_4(F)$ is a theta lift from a generic, irreducible representation $\pi$ with a Shalika model.
In a forthcoming paper, we really construct a Whittaker `newform' of a generic, irreducible representation of $\GSp_4(F)$ using the above $J_\pi$, and show that the inequality $\cf_\pi \ge 2e$ holds for any $\pi$.
We think that this supports the validity of our Shalika form $J_\pi$ and open compact subgroup $\K(\cf_\pi)$.

This article is organized as follows.
In sec. \ref{sec:pre}, we introduce several fundamental terminolgies used throughout. 
In sec. \ref{sec:PSM}, we show the uniqueness of pre-Shalika models, and devote to compute the values of  a `canonical' pre-Shalika form, which plays an essentially important role.
These notation are defined for representations of the mirabolic subgroups.
The method to determine all values of $J^{new}$ on the mirabolic subgroup $P_{n}$ is also given.
In sec. \ref{sec:Wm}, using the results of the previous section, we give a Whittaker model of supercuspidal, generic, irreducible representation $\pi$ of $G_n$ realized in $\Sb_\pi(\chi)$.
In sec. \ref{sec:HO}, we prove a proposition of the previous section by a Hecke theory.
In sec. \ref{sec:ZI}, we construct $J_\pi$, and show that it has the above nice properties.

{\bf Notation:} 
Throughout $F$ denotes a $p$-adic field with ring $\oo$ of integers.
Let $\p = \Vp \oo$ denote the prime ideal of $\oo$ and $q$ the cardinality of the residue filed.
Let $\psi: F \to \C^\tm$ denote a nontrivial additive character, where $\C^\tm$ indicates the multiplicative group of the complex number field $\C$.
Let $\chi: F^\tm \to \C^\tm$ denote a continuous homomorphism, and $e$ its (order of) conductor.
Let $o(x)$ denote the $p$-adic order of $x \in F$, and $\nu(x)$ the $p$-adic norm: $\nu(x) = q^{-o(x)}$.

If $G$ is a group, $h,g$ are elements in $G$ and $H$ is a subgroup of $G$, then we write $h^g$ and $H^g$ for $g h g^{-1}$ and $\{h^g \mid h \in H \}$, respectively.
We use the notation for the groups: 
\begin{align*}
G_r &= \GL_r(F) \\
B_r & = \mbox{the standard Borel subgroup of $G_r$}\\
D_r &= \mbox{the diagonal matrices in $G_r$} \\
N_r &= \mbox{the unipotent matrices in $B_r$} \\
P_r &= \{(p_{ij}) \in G_r \mid p_{rj} = 0 \ \mbox{for $j < r$}, \ p_{rr} = 1 \} \\
U_r &= \{(u_{ij}) \in N_r \mid u_{ij} = \delta_{ij} \ \mbox{for $j <r$} \} \\
\bar{U}_r & = \mbox{the opposite of $U_r$} \\
K_r&= \GL_r(\oo) \\
\Sf_r & = \mbox{the symmetric group of degree $r$}.
\end{align*}  
For a permutation $w$ of the set $\{1, \ldots, r\}$, we also denote by $w$ the permutation matrix in $G_r$ defined by 
\begin{align*}
w[g_{i,j}]w^{-1} = [g_{w(i),w(j)}],
\end{align*} 
and identify $\Sf_r$ a subgroup of $G_r$ naturally.
When a positive integer $l$ is clear from the context, let $\acute{*}: G_r \to G_{r+l}$ denote the embedding 
\begin{align*}
g \longmapsto \ac{g} = \begin{bmatrix}
g &  \\
& 1_{l}
\end{bmatrix},
\end{align*} 
and for a function $f$ on a subset of $G_{r+l}$ containing $\ac{G}_r$, let $\ac{f}$ denote the function on $G_r$ by the pullback, where $1_{l} \in G_r$ indicates the identity.
\section{Preparation}\label{sec:pre}
Let $r$ be an integer larger than $1$.
Fix $\psi$.
Denote also by $\psi$ the homomorphism
\begin{align}
 N = N_r \ni n \longmapsto \prod_{1 \le i \le r-1} \psi(n_{i,i+1}) \in \C^\tm. \label{eq:psiex}
\end{align} 
Let $\pi$ be an irreducible representation of $G= G_r$.
It is known that $\Hom_{G}(\pi, \In_N^{G} \psi)$, the $\C$-space of Whittaker models, is at most one-dimensional, where $\In$ indicates the induction functor.
If $\pi$ has a Whittaker model,  then $\pi$ is called {\it generic}.
In this case, let $\Wb_\pi$ denote the image of $\pi$ under it, and call $W \in \Wb_\pi$ {\it Whittaker forms of $\pi$}.
For a positive integer $m$, let $\Gamma_r(m) = \Gamma(m) \subset G_r$ be the open compact subgroup 
\begin{align}
\{k \in K_r \mid k_{r,1}, \ldots,  k_{r,r-1} \in \p^{m} \}. \label{def:Gc}
\end{align} 
Assume that 
\begin{align}
\psi(\oo) = \{1\} \neq \psi(\p^{-1})  \label{ass:psi}.
\end{align}
By the work of Jacquet, Piatetski-Shapiro, and Shalika \cite{JPSS}, there exists a unique $W^{new} \in \Wb_\pi$ such that 
\begin{align*}
W^{new}(1_r) = 1, \ \ \pi(k) W = \w_\pi(k_{rr})W, k \in \Gamma(\cf_\pi),
\end{align*} 
where $\cf_\pi$ and $\w_\pi$ indicate the conductor and central character of $\pi$, respectively. 
We call $W^{new}$ the {\it canonical Whittaker newform} of $\pi$.

If $r$ is even, then {\it the Shalika subgroup $S_r =S \subset G_r$} is defined to be the subgroup consisting of the matrices 
\begin{align*}
s = \begin{bmatrix}
a &  \\
&a
\end{bmatrix}
\begin{bmatrix}
1_{r/2}& b \\
& 1_{r/2}
\end{bmatrix}.
\end{align*} 
For $\chi$, let $\chi_\psi: S \to \C^\tm$ denote the homomorphism  
\begin{align}
s \longmapsto \chi \circ \det(a) \psi(\tr(b)). \label{def:chipsi}
\end{align} 
F. Chen, and B. Sun \cite{C-S} showed that $\Hom_{G}(\pi, \In_S^G(\chi_\psi))$, the $\C$-space of Shalika models, is at most one-dimensional.
If $\pi$ has a Shalika model, then let $\Sb_\pi(\chi)$ denote the image of $\pi$ under it, and call $J \in \Sb_\pi(\chi)$ {\it Shalika forms of $\pi$ relevant to $\chi$}.
By definition, 
\begin{align}
J(s g) = \chi_\psi(s) J(g). \label{eq:propSh}
\end{align} 
By the above Whittaker newform theory, there exists a unique $J \in \Sb_\pi(\chi)$ up to constant multiples such that 
\begin{align*}
\pi(k) J = \w_\pi(k_{rr})J, k \in \Gamma(\cf_\pi),
\end{align*} 
which are called {\it the Shalika newforms of $\pi$ (relevant to $\chi$)}.

Let
\begin{align*}
S_{r}^\circ = S_{r} \cap P_{r}.
\end{align*} 
For an irreducible smooth representation $\tau$ of $P_r$, we call the $\C$-space 
\begin{align*}
\Hom_{P_r}(\tau, \In_{S_r^\circ}^{P_r}(\chi_\psi)) \simeq \Hom_{S_r^\circ}(\tau, \chi_\psi)
\end{align*} 
{\it the pre-Shalika models of $\tau$ relevant to $\chi$}, and will show that it is at most one-dimensional in the next section.
If $\tau$ has a pre-Shalika model, then let $\Ib_\tau(\chi)$ denote its image of $\tau$, and call vectors in $\Ib_\tau(\chi)$ {\it pre-Shalika forms of $\tau$ relevant to $\chi$}.
\section{Pre-Shalika model}\label{sec:PSM}
In this section, firstly we show that, for an arbitrary $\chi$, the irreducible smooth representation
\begin{align}
\psi_n := \cI_{N_{n}}^{P_{n}}(\psi) \label{def:psim}
\end{align}
has a unique pre-Shalika model up to constant multiples, where $\cI$ indicates the compact induction functor.
By this uniqueness, the restriction to $P_n$ of a Shalika form of a supercuspidal representation coincides with a pre-Shalika form of $\psi_n$, and vice versa (c.f. Prop. \ref{prop:SPS}, \ref{prop:asaV}).
Secondly, we compute the support of a $P_n(\oo)$-invariant pre-Shalika form, playing the essentially important role in this article, of $\psi_n$ constructed by the pre-Shalika model, where 
\begin{align*}
P_n(\oo) = \{p \in P_n \mid p_{ij} \in \oo \}.
\end{align*} 
\subsection{Uniqueness of pre-Shalika models}
For an $l$-group $G$ (c.f. \cite{B-Z}), we denote by $\Alg(G)$ (resp. $\Ir(G)$) the category of smooth (resp. smooth irreducible) $\C[G]$-modules.
For $\tau \in \Alg(P_r) \cup \Alg(G_r)$, we call {\it the lift of $\tau$ to $P_l$} with $l > r$ the representation 
\begin{align*}
\tau_l := 
\begin{cases}
\Psi^{l-r} (\tau) & \mbox{if $\tau \in \Alg(P_r)$} \\
\Psi^{l-r-1}\circ \Upsilon(\tau) & \mbox{if $\tau \in \Alg(G_r)$.} 
\end{cases}
\end{align*}
where $\Psi:\Alg(P_s) \to \Alg(P_{s+1})$ and $\Upsilon:\Alg(G_s) \to \Alg(P_{s+1})$ are the functors defined by 
\begin{align*}
&\Psi: \tau \longmapsto \cI_{\ac{P}_{s}U_{s+1}}^{P_{s+1}} (\tau \lt \psi) \\
& \Upsilon: \tau \longmapsto \tau \lt \1_{U_{s+1}}.
\end{align*} 
Here $\tau \lt \psi$ (resp. $\tau \lt \1_{U_{s+1}}$) indicates the representation sending $\ac{p} u \in \ac{P}_sU_{s+1}$ to $\psi(u)\tau(p)$ (resp. $\tau(p)$).
Abbreviate $\In_{S^\circ_{n}}^{P_{n}} (\chi_\psi)$ as $\Ib^{n}(\chi)$.
\begin{prop}\label{prop:redG}
Let $\tau \in \Ir(G_{r})$.
Let $r$ be a positive integer.
For an even integer $n > r$, we have the followings.
\begin{enumerate}[i)]
\item If $r$ is odd, then $\tau_{n}$ has no pre-Shalika model.
\item If $r$ is even, then  
\begin{align*}
\dim \Hom_{P_{n}}(\tau_{n}, \Ib^{n}(\chi)) = \dim \Hom_{G_{r}}(\tau, \In_{S_{r}}^{G_r}(\nu^{\frac{n- r}{2}}\chi_\psi)).
\end{align*} 
\end{enumerate}
\end{prop}
\nid
This follows from the induction on $n$, and the following three lemmas.
\begin{lem}
For $\tau \in \Alg(G_{r})$ with $r$ odd, $\tau_{r+1}$ does not have a pre-Shalika model relevant to any $\chi$.
\end{lem}
\begin{proof}
Let $\lambda \in \Hom_{P_{r+1}}(\tau_{r+1},\Ib^{r+1}(\chi))$, and $f \in \tau_{r+1}$.
By the definition of $\Upsilon$ and that of the pre-Shalika space, 
\begin{align*}
\lambda(f) = \lambda(\tau_{r+1}(u) f) = \psi(u_{\frac{r+1}{2},r+1})\lambda(f), \ u \in U_{r+1}.
\end{align*}
Hence $\lambda(f)$ vanishes, and the assertion.
\end{proof}
\begin{lem}\label{lem:inGchi}
For $\tau \in \Alg(G_{n})$ with $n$ even, and an arbitrary $\chi$, 
\begin{align*}
\dim \Hom_{P_{n+2}}(\tau_{n+2}, \Ib^{n+2}(\chi))= \dim \Hom_{G_{n}}(\tau, \In_{S_{n}}^{G_{n}}(\nu\chi_\psi)).
\end{align*} 
\end{lem}
\begin{lem}\label{lem:delta2}
For $\tau \in \Alg(P_{n})$ with $n$ even, and an arbitrary $\chi$, 
\begin{align*}
\dim \Hom_{P_{n+2}}(\tau_{n+2}, \Ib^{n+2}(\chi)) = \dim \Hom_{P_{n}}(\tau, \Ib^{n}(\nu\chi)).
\end{align*} 
\end{lem}
\nid
To show the last two lemmas, we use the distributional technique of loc. cit. for $l$-spaces $X$.
For a $\C$-vector space $V$, let $\Ss(X,V)$ denote the space of all Schwartz functions on $X$ with values in $V$.
Linear functionals $T$ on $\Ss(X,V)$ are called $V$-distributions on $X$.
Additionally, if $X$ is an $l$-group, and $T$ is right (resp. left) invariant under an open compact subgroup, we say $T$ is {\it locally constant on the right (resp. left)}.
\begin{prop}\label{prop:distprod}
Let $G$ be an $l$-group, and $V$ a vector space over $\C$.
Let $T$ be a $V$-distribution on $G$.
We have the followings:
\begin{enumerate}[i)]
\item If $T$ is right (resp. left) invariant, then $T$ is the product of a right (resp. left) Haar measure on $G$, and a linear functional on $V$.
\item If $T$ is locally constant on the  right (resp. left), then $T$ is the product of a right (resp. left) Haar measure on $G$, and a $V^*$-valued continuous function on $G$, where $V^*$ indicates the full-dual of $V$.
\end{enumerate}
\end{prop}
\begin{proof}
Suppose that $T$ is right invariant.
Let $\dr_r g$ denote a right Haar measure on $G$.
Let $\{ N_\alpha \mid \ \alpha \in \Af\}$ be a fundamental system of neghbourhoods of the identity consisting of open compact subgroups of $G$.
For $v \in V$, let $\vp^\alpha_v = \Ch(g; N_\alpha) v \in \Ss(G,V)$, where $\Ch$ indicates the characteristic function.
For each $\alpha \in \Af$, define $v_\alpha^* \in V^*$ by 
\begin{align*}
\langle v_\alpha^*, v \rangle = \vol(N_\alpha)^{-1} T(\vp^\alpha_v), \ v \in V.
\end{align*} 
Therefore $v^*_\alpha = v^*_\beta$ if $N_\alpha$ contains $N_\beta$, since $T(\vp^\alpha_v)$ equals $[N_\alpha:N_\beta]T(\vp^\beta_v)$ by the right invariance property of $T$.
This implies that $v_\alpha^*$ is independent from the choice of $\alpha$.
Since each $\vp \in \Ss(G,V)$ is right invariant under some $N_\alpha$, we may express $\vp$ by a finite sum of some right translations of $\vp_{v_i}^\alpha$ with some $v_i$'s.
By the right invariance property of $T$ again, we have
\begin{align*}
 T(\vp) &= T(\sum_i \vp_{v_i}^\alpha) = \sum_i\vol(N_\alpha) \langle v_\alpha^*, v_i \rangle \\
 &=\sum_i \int_G \langle v_\alpha^*, \vp^\alpha_{v_i}(g) \rangle  \dr_rg = \int_G \langle v_\alpha^*, \vp(g) \rangle  \dr_rg.
\end{align*} 
This is the former assertion.
Now, the latter one follows from the proof of Prop. 1.28. of loc. cit.
\end{proof}
\nid
Let $G$ be an $l$-group.
Let $Q_0$ and $U$ be its closed subgroups such that $Q_0 \cap U = \{1 \}$, and $Q_0$ normalizes $U$.
Set
\begin{align*}
Q = Q_0 U,
\end{align*}  
which is a closed subgroup of $G$.
Let $\xi:U \to \C^\tm$ be a continuous homomorphism stabilized by $Q_0$.
Let $H \subset G$ be a closed subgroup, and $\rho: H \to \C^\tm$ a continuous homomorphism.
Assume that
\begin{align}
QH = \{g \in G \mid \xi(h^g) = \rho(h) \ \  \mbox{for all $h \in H \cap U^{g^{-1}}$} \}. \label{eq:QH}
\end{align}
Observe the last set is left $Q$-, right $H$-invariant, closed, and an $l$-space in the induced topology.
Additionally, assume that 
\begin{align}
Q \cap H = (Q_0 \cap H)(U \cap H). \label{eq:QHst}
\end{align} 
\begin{prop}
With the above assumptions, for an arbitrary $\pi \in \Alg(Q_0)$, it holds that
\begin{align*}
\dim \Hom_H(\cI_Q^G(\pi \lt \xi),\rho) = \dim \Hom_{Q_0 \cap H}(\pi,  \Delta_{Q \cap H}\Delta_H^{-1} \rho),
\end{align*}
where $\Delta_{Q \cap H}, \Delta_H$ indicate the modular characters on the groups.
\end{prop}
\begin{proof}
Our proof is a modification of that for Prop. 1 of \cite{K}.
Abbreviate $\pi \lt \xi$ and $\Delta_{Q \cap H}\Delta_H^{-1}$ as $\pi_\xi$ and $\Delta$, respectively.
We will construct a linear map from the $\C$-space $\Hom_H(\cI_Q^G(\pi_\xi),\rho)$ to $\Hom_{Q_0 \cap H}(\pi, \Delta \rho)$.
Denote the representation space of $\pi$ by $V_\pi$, on which $Q$ also acts by $\pi_\xi$.
For $\phi \in \Ss(G, V_\pi)$, define $f_\phi \in \cI_Q^G(\pi_\xi)$ by 
\begin{align*}
f_\phi(g) = \int_{Q} \pi_\xi(q^{-1})\phi(qg )\dr_rq.
\end{align*} 
The linear map $\phi \mapsto f_\phi$ is bijective.
For $\mu \in \Hom_H(\cI_Q^G(\pi_\xi),\rho)$, define a $V_\pi$-distribution $T_\mu$ on $G$ by $T_\mu(\phi) = \mu(f_\phi)$.
Observe that $T = T_\mu$ satisfies:  
\begin{align}
\begin{split}
T \circ R(h) &= \rho(h) T \ \ (h \in H), \\
T \circ L(q) &= \Delta_{Q}(q) T \circ \pi_\xi(q^{-1}) \ \ (q \in Q).
\end{split}\label{eq:Vdist}
\end{align}  
Consider the double coset space $Q \bs G / H$.
By (\ref{eq:QH}), if $g$ does not lie in $QH$, then it holds that $\rho(h) \neq \xi(h^g)$ for some $h \in H$, and $T(g)$ is zero by (\ref{eq:Vdist}). 
Therefore, the support of $T$ is contained in $QH$, and we may regard $T$ as a $V_\pi$-distribution on the closed subset $QH$ (an $l$-space by (\ref{eq:QH})).

Let $\vp \in \Ss(Q \tm H, V_\pi)$.
Define $\ol{\vp} \in \Ss(QH, V_\pi)$ by 
\begin{align*}
\ol{\vp}(q^{-1}h) = \int_{Q \cap H} \Delta_Q(aq) \pi_\xi((aq)^{-1})\rho(ah) \vp(aq, a h) \dr_ra.
\end{align*} 
The linear map $\vp \mapsto \ol{\vp}$ is bijective.
Therefore any $V_\pi$-distribution $T'$ on $Q \tm H$ is derived from a $V_\pi$-distribution $T$ on $QH$ satisfying (\ref{eq:Vdist}) by setting $T'(\vp) = T(\ol{\vp})$.
For $q' \in Q, h' \in H$, we compute
\begin{align*}
&\ol{R(q',h') \vp}(q^{-1}h) 
= \int_{Q \cap H} \Delta_{Q}(aq) \pi_\xi((aq)^{-1})\rho(ah) \vp(aqq', a hh') \dr_ra \\
&= \frac{\pi_\xi(q')}{\Delta_{Q}(q')\rho(h')}\int_{Q \cap H} \Delta_{Q}(aqq') \pi_\xi((aqq')^{-1})\rho(ahh') \vp(aqq', a hh') \dr_ra \\
&= \frac{\pi_\xi(q')}{\Delta_{Q}(q')\rho(h')} L(q')R(h')\ol{\vp}(q^{-1}h). 
\end{align*} 
Therefore $T'$ is right invariant by (\ref{eq:Vdist}).
By Prop. \ref{prop:distprod}, there exists a linear functional $\mu'$ on $V_\pi$ such that 
\begin{align}
T'(\vp) = \int _H \int_{Q} \langle \mu', \vp(q,h) \rangle \dr_rq \dr_rh. \label{eq:mu'}
\end{align} 
For $b \in Q \cap H$, we compute
\begin{align*}
& \ol{L(b,b)\vp}(q^{-1}h) = \int_{Q \cap H} \Delta_Q(aq) \pi_\xi((aq)^{-1})\rho(ah) \vp(b^{-1}aq, b^{-1}a h) \dr_ra \\
&= \Delta_{Q \cap H}(b) \int_{Q \cap H} \Delta_Q(baq)\pi_\xi((aq)^{-1}b^{-1})\rho(bah) \vp(aq, a h) \dr_ra \\
&= \Delta_{Q \cap H}(b) \Delta_Q(b) \rho(b)\ol{\pi_\xi(b^{-1})\vp}(q^{-1}h).
\end{align*} 
Thus 
\begin{align*}
T' \circ L(b,b) = \Delta_{Q \cap H}(b) \Delta_Q(b) \rho(b) T' \circ \pi_\xi(b^{-1}).
\end{align*} 
Now from (\ref{eq:mu'}), 
\begin{align*}
\Delta_H(b) \langle \mu', \vp(q,h) \rangle = \Delta_{Q \cap H}(b) \rho(b) \langle \mu', \pi_\xi(b^{-1})\vp(q,h) \rangle.
\end{align*}  
This means that $\mu'$ lies in $\Hom_{Q \cap H}(\pi_\xi, \Delta \rho)$. 
By the restriction to $Q_0$, we obtain a $\mu'' \in \Hom_{Q_0 \cap H}(\pi, \Delta\rho)$. 
The linear map $\mu' \mapsto \mu''$ is bijective by (\ref{eq:QHst}).
We have constructed the desired map $\mu \mapsto \mu''$.
One can reverse the above steps, and easily find that this map is bijective.
\end{proof}
\begin{rem}
The spaces $\Hom_H(\cI_Q^G(\pi_\xi),\rho)$ and $\Hom_{Q \cap H}(\pi_\xi, \Delta \rho)$ are both zero, if $\Delta|_{U \cap H} \not\equiv \1$.
\end{rem}
Let $n = 2m$ be an even integer.
By the Frobenius duality, 
\begin{align*}
\Hom_{P_{n+2}}(\tau_{n+2}, \Ib^{n+2}(\chi)) &\simeq \Hom_{S_{n+2}^\circ}(\tau_{n+2}, \chi_\psi), \\
\Hom_{P_{n}}(\tau, \Ib^{n}(\nu\chi)) &\simeq \Hom_{S_{n}^\circ}(\tau, \nu\chi_\psi).
\end{align*} 
We will prove Lemma \ref{lem:delta2} by showing the spaces in the right hands are same dimensional.
By the previous proposition, it suffices to check (\ref{eq:QH}), (\ref{eq:QHst}) under the situation considered.
Set $G = P_{n+2}$.
Let $h_{n+1} \in \Sf_{n+1}$ be the permutation 
\begin{align}
\left(\begin{array}{cccccccc}
1& \cdots & m & m+1 & m+2  & m+3 & \cdots & n+1 \\
1 & \cdots & m & 2m+1 & m+1 & m+2 & \cdots & n 
\end{array}\right). \label{def:hn+1}
\end{align} 
Abbreviate
\begin{align}
\eta = \ac{h}_{n+1} \in P_{n+2}. \label{def:h}
\end{align} 
Set $H = \eta S_{n+2}^{\circ}\eta^{-1}$.
Let $\rho: H \to \C^\tm$ be the homomorphism: 
\begin{align*}
h \longmapsto \chi_\psi(\eta^{-1}h\eta ).
\end{align*}
If we write a typical $s \in S_{n+2}^\circ$ as
\begin{align*}
\begin{bmatrix}
A & {}^t \alpha & X &{}^t \beta  \\
& 1 & \gamma &  y \\
& &A & {}^t \alpha  \\
& & & 1 \\
\end{bmatrix},
\end{align*} 
where $A,X$ are $m \tm m$ matrices, $\alpha, \beta, \gamma$ are $m$-dimensional row vectors, and $y$ is an element of $F$, then 
\begin{align*}
s^\eta &= \begin{bmatrix}
A & X & {}^t\alpha & {}^t\beta \\
& A & & {}^t\alpha \\
& \gamma &1 & y \\
& & & 1 \\
\end{bmatrix} \\
\rho(s^\eta) &= \chi \circ \det(A)\psi(y + \tr(X)).
\end{align*} 
Set $Q_0 = \ac{P}_{n} \subset G_{n+2}$, and $U = U_{n+1}U_{n+2}$,
Set $\xi = \psi|_U$. 
Observe that
\begin{itemize}
\item $\tau_{n+2} = \cI_Q^G(\tau \lt \xi)$.
\item $\ac{S}_{n}^\circ = H \cap Q_0$.
\item $\rho(\ac{t}) = \chi_\psi(t)$ for $t \in S_{n}$.
\item $\Delta_{Q \cap H}\Delta_H^{-1}(\ac{t})=|\det(A)|$ for 
$t = \begin{bmatrix}
A & * \\
& A  
\end{bmatrix} \in S_{n}$.
\end{itemize}
It is easy to check (\ref{eq:QHst}).
For (\ref{eq:QH}), it suffices to see that any matrix $p$ in the RHS of (\ref{eq:QH}) satisfies: 
\begin{align*}
p_{n,1} = \cdots = p_{n, m} = p_{n+1,1} = \cdots = p_{n+1, m} = 0.
\end{align*} 
Let $E_{i,j}$ denote the $i$-th row and $j$-th column matrix unit.
If $p_{n+1,i} \neq 0$ for $i \in \{1, \ldots, m\}$, then for some $h = 1_{n+2} + x E_{i,n+2} \in H$,  
\begin{align*}
\rho(h) = 1 \neq \psi(x p_{n+1,i}) = \xi(h^p).
\end{align*} 
Hence, $p_{n+1,1} = \cdots = p_{n+1, m} = 0$.
Now we may assume that $p$ lies in $\ac{P}_{n+1}U_{n+2}$, since the RHS of (\ref{eq:QH}) is right $H$-invariant.
If $p_{n,i} \neq 0$ for $i \in \{1, \ldots, m\}$, then for some $h'= 1_{n+2} + y E_{i,n+1} + yE_{i+ m,n+2} \in H$, 
\begin{align*}
\rho(h') = 1 \neq \psi(y p_{n,i}) = \xi(h'^p).
\end{align*} 
Hence, $p_{n,1} = \cdots = p_{n, m} = 0$.
Now (\ref{eq:QH}) is checked, and the proof of the lemma is completed.
The proof of Lemma \ref{lem:inGchi} is similar and omitted.

Similar to Prop. \ref{prop:redG}, the following holds.
\begin{prop}\label{prop:red0}
For an arbitrary $\chi$, $\psi_{n}$ has a unique up to constant multiples pre-Shalika model relevant to $\chi$. 
\end{prop}
\begin{proof}
Since $\psi_{n}$ equals $\Psi^{n-2}(\psi_2)$,  the assertion is reduced to the evaluation of the dimension of $\Hom_{P_{2}}(\psi_2, \Ib^{2}(\chi))$ by Lemma \ref{lem:delta2}. 
Consider the space of corresponding $\C$-distributions on $P_2$, and the corresponding double coset space $N_2 \bs P_2 /N_2$.
As a complete system of its representatives, using $\{\ac{t} \in P_2 \mid t \in F^\tm \}$, we find that all supports of the distributions are contained in $N_2$.
\end{proof}
\nid
Since any irreducible smooth representation of $P_n$ is equivalent to $\psi_n$, or a lift from $\Ir(G_{r}), r < n$ (c.f. \cite{B-Z}), we obtain:
\begin{thm}\label{thm:1dimprtau}
An irreducible smooth representation of $P_{n}$ has no or a unique up to constant multiples pre-Shalika model relevant to $\chi$.
\end{thm}

Now let $(\pi,V) \in \Ir(G_{n})$ be generic.
There exists a Jordan-H\"{o}lder sequence of smooth $P_{n}$-modules $V_l \subset \cdots \subset V_0 = V$ with the following properties (c.f. \cite{B-Z}):
\begin{itemize}
\item $V_l$ is equivalent to $\psi_n$.
\item each $V_i/V_{i + 1}$ is equivalent to some lift from $\Ir(G_r), r < n$.
\item $\pi$ is supercuspidal, if and only if $V = V_l$.
\end{itemize}
Assume that 
\begin{align}
\mbox{$V_l$ contains a Shalika form relevant to $\chi$ not vanishing at $1_n$}. \label{ass:J1nz}
\end{align} 
Of course, if $\pi$ is supercuspidal, then this condition is empty.
\begin{prop}\label{prop:SPS}
Under (\ref{ass:J1nz}), for any pre-Shalika form of $\psi_n$ relevant to $\chi$, there exists a Shalika form in $V_l$ whose restriction to $P_{n}$ coincides with it.
\end{prop}
\begin{proof}
By (\ref{ass:J1nz}), for a Shalika form in $V_l$, its restriction to $P_n$ is nontrivial, and can be regarded as a pre-Shalika form of $\psi_n$.
The assertion follows from the irreducibility of $V_l \simeq \psi_n \simeq \Ib_{\psi_n}(\chi)$.
\end{proof}
\begin{prop}\label{prop:asaV}
Assume that all $V_i/V_{i+1}$ has no pre-Shalika model relevant to $\chi$.
Then, (\ref{ass:J1nz}) is satisfied.
Further, for a Shalika form $J$, the followings are equivalent.
\begin{enumerate}[i)]
\item $J$ does not vanish on $P_n$.
\item $J$ lies in $V_l$.
\end{enumerate}
In this case, $J|_{P_n}$ coinsides with a pre-Shalika form of $\psi_n$.
\end{prop}
\begin{proof}
Let $V_i^0$ be the $P_n$-submodule $\{J \in V_i \mid J|_{P} \equiv 0 \} \subset V_i$.
Take the minimal $V_r$ containing a Shalika form which does not vanish at the identity.
By definition, if $r \neq l$, then we have $V_{r+1} = V_{r+1}^0 \subset V_r^0$ and $V_r/V_{r+1}$ has a pre-Shalika model, conflicting with the assumption.
Hence the first assertion. 
This argument also implies the equivalence of i), ii).
The last assertion follows from the proof of Prop. \ref{prop:SPS}.
\end{proof}
\nid
We will give an easy sufficient condition for the lack of Shalika models.
\begin{lem}\label{lem:notauram}
If $\tau \in \Ir(G_r)$ is ramified, then any lift of $\tau$ has no nontrivial $P_{r}(\oo)$-invariant vector.
\end{lem}
\begin{proof}
Let $f \in \tau_l$ with $l > r$ be $P_{l}(\oo)$-invariant.
Let 
\begin{align*}
T = \{\diag(t_1, \ldots, t_l) \mid t_{1} = \cdots = t_{r} = 1 \} \subset D_l.
\end{align*}
By the inclusive relation $P_l \subset N_l \ac{G}_r T  P_{l}(\oo)$, and the definition of $\tau_l$, it suffices to show that $f$ vanishes on $\ac{G}_r T$.
The restriction to $\ac{G}_r$ of the right translation of $f$ by $t \in T$ is $K_r$-invariant.
However, it is identically zero since $\tau$ is ramified.
\end{proof}
\begin{lem}\label{lem:unrnoS}
If $\chi$ is ramified, then any unramified $\pi \in \Ir(G_{n})$ with $n$ even has no Shalika model relevant to $\chi$.
\end{lem}
\begin{proof}
Assume that $\pi$ has a Shalika model relevant to ramified $\chi$.
Then there exists a nontrivial $K_{n}$-invariant $J \in \Sb_\pi(\chi)$.
By the Iwasawa decomposition of $G_{n}$, $J|_{B_{n}} \not\equiv 0$.
By (\ref{eq:propSh}), $J|_{\ac{B}_{n/2}} \not\equiv 0$.
By (\ref{eq:propSh}) and the Cartan decomposition of $G_{n/2}$, we may assume that $J(\ac{d}) \neq 0$ for some $d \in D_{n/2}$.
By (\ref{eq:propSh}) and the $K_{n}$-invariance property of $J$,  
\begin{align*}
\chi(t) J(\ac{d}) = J(\ac{t} \ac{d}) =  J(\ac{d} \ac{t}) =  J(\ac{d})
\end{align*} 
for any $t \in \oo^\tm$.
Since $\chi$ is ramified, it follows that $J(\ac{d}) = 0$, a contradiction.
\end{proof}
\nid
Combining the above lemmas and Prop. \ref{prop:redG}, 
\begin{prop}\label{prop:genPSur}
Any lift of $\tau \in \Ir(G_{r})$ has no pre-Shalika model relevant to $\chi$, if one of $\tau, \chi$ is ramified and another is unramified.
\end{prop}
\subsection{Essential pre-Shalika form}\label{subsec:spsf}
Let $n = 2m$ be an even integer. 
If $\tau \in \Ir(P_n)$ has a Shalika model relevant to $\chi$, then that of $\tau_{n+2}$ relevant to $\nu^{-1}\chi$ is given as follows.
Let $\dr A$ be the right $G_m$-invariant measure on $P_m \bs G_m$ normalized so that $\vol(P_m \bs P_mK_m) = 1$.
Let $\dr x$ be the Haar measure on $F^m$ normalized so that $\vol(\oo^m) = 1$.
For $\vp \in \Ib_\tau(\chi)_{n+2}$, consider the function $\theta_\vp$ on $P_{n+2}$ defined by the integral: 
\begin{align*}
\theta_\vp(p) &= \int_{P_m \bs G_{m}} \int_{F^m}\chi \circ \det(A)^{-1} \vp(\eta \begin{bmatrix}
A & & &  \\
&1 & x & \\
& & A &  \\
& & &1  \\
\end{bmatrix}p) \dr x \dr A \\
&= \int_{P_m \bs G_{m}} \int_{F^m}  \chi \circ \det(A)^{-1} \vp(\begin{bmatrix}
A & & &  \\
&A &  & \\
& x & 1 &  \\
& & &1  \\
\end{bmatrix}\eta p) \dr x \dr A
\end{align*} 
where $\eta$ is the matrix (\ref{def:h}).
Since the integrated function as a function on $G_m \ltimes F^{m}$ is compactly supported modulo $P_m$, the integrals can be exchanged and converge.
Obviously, $\theta_\vp$ lies in $\Ib^{n+2}(\nu^{-1}\chi)$.
Denote by $\theta_{n+2}^\chi$ this linear map $\vp \mapsto \theta_\vp$.
We will see the nontriviality of $\theta_{n+2}^\chi$.
We may assume that (\ref{ass:psi}). 
Take a $\xi \in \Ib_\tau(\chi)$. 
For a nonnegative integer $l$, define the open compact subgroup 
\begin{align*}
C_n(l) = \{p \in P_{n}(\oo) \mid p \equiv 1_{n} \pmod{\p^l} \} \subset P_{n}.
\end{align*}
Since $\Ib_\tau(\chi)$ is smooth, $\xi$ is invariant under $C_{n}(l)$ for a sufficiently large $l$.
Define the $C_{n+2}(l)$-invariant $\vp_\xi \in \Ib_\tau(\chi)_{n+2}$ by 
\begin{align*}
\vp_\xi(p) = 
\begin{cases}
\psi(u) \xi(g) & \mbox{if $p \in  u \ac{g} C_{n+2}(l)$ for $g \in P_{n}, u \in U_{n+1}U_{n+2}$} \\
0 & \mbox{otherwise.} 
\end{cases}
 \end{align*}
Denote by $\phi_n^{(l)}$ this construction (linear map) $\xi \mapsto \vp_\xi$
for $C_n(l)$-invariant $\xi \in \Ib_{\tau}(\chi)$. 
 \begin{prop}\label{prop:welldefJW}
 We have the followings.
 \begin{enumerate}[i)]
\item The above $\vp_\xi$ is well-defined.
\item If $t \not\in \oo^\tm$, then
\begin{align*}
\vp_\xi(\begin{bmatrix}
* & * &*  \\
&t & * \\
& & 1
\end{bmatrix}) = 0.
\end{align*} 
\item $\theta_{\vp_\xi}(\eta^{-1})$ is a nonzero constant multiple of $\xi(1_n)$.
\item Let $T$ be a subset of $P_n$.
Define the function $\xi'$ on $P_n$ by
\begin{align*}
\xi'(g) = \int_T \xi(t g) \dr t.
\end{align*} 
Then, for $u \in U_{n+1}U_{n+2}$, 
\begin{align*}
\int_{T} \vp_\xi(t u\ac{g}) \dr t = \psi(u)\vp_{\xi'}(\ac{g}).
\end{align*} 
\end{enumerate}
\end{prop}
\begin{proof}
i) Suppose that $\ac{g} u C_{n+2}(l) = \ac{h} v C_{n+2}(l)$ for $g,h \in P_{n}$ and  
\begin{align*}
u = \begin{bmatrix}
1_n & {}^tx & {}^t y  \\
& 1& z \\
& &1 
\end{bmatrix}, 
v = \begin{bmatrix}
1_n & {}^t x' & {}^t y'  \\
& 1& z' \\
& &1 
\end{bmatrix} \in U_{n+1} U_{n+2},
\end{align*}
where $x,x',y,y' \in F^n$, $z,z' \in F$.
Then, $v^{-1} \ac{h}^{-1}\ac{g} u$ lies in $C_{n+2}(l)$, and is in the form of 
\begin{align*}
\begin{bmatrix}
h^{-1}g & {}^tw &* \\
& 1& z-z'  \\
& &1 
\end{bmatrix}, w_n = x_n -x_n'.
\end{align*} 
Thus, $x_n -x_n', z-z'$ lie in $\p^l$, and $h^{-1}g$ lies in $C_{n}(l)$. \\
ii) Follows from our construction of $\vp_\xi$ and the observation:
\begin{align*}
\begin{bmatrix}
* & &  \\
& 1&  \\
& & 1
\end{bmatrix} 
\begin{bmatrix}
1_n & * &*  \\
&1 & * \\
& & 1
\end{bmatrix} \begin{bmatrix}
* & * &*  \\
&t & * \\
& & 1
\end{bmatrix} \not\in C_{n+2}(l).
\end{align*} 
iii) Let $K_m(l) = \{k \in K_m \mid k \equiv 1 \pmod{\p^l} \}$.
By the definition of $\vp_\xi$, 
\begin{align*}
\vp_\xi(\begin{bmatrix}
A & & &  \\
&A &  & \\
& x & 1 &  \\
& & &1  \\
\end{bmatrix}) = 0
\end{align*} 
unless $x \in (\p^l)^m$ and $A \in P_m K_m(l)$. 
Therefore, 
\begin{align*}
\theta_{\vp_\xi}(\eta^{-1}) &= \int_{P_m \bs P_m K_m(l)} \int_{(\p^l)^m}  \chi \circ \det(A)^{-1} \vp_\xi(\begin{bmatrix}
A & & &  \\
&A &  & \\
& x & 1 &  \\
& & &1  \\
\end{bmatrix}) \dr x \dr A \\
&= q^{-lm} \int_{P_m \bs P_m K_m(l)}  \chi \circ \det(A)^{-1} \xi(A) \dr A \\
&= q^{-lm} \vol(P_m \bs P_m K_m(l)) \xi(1_n). 
\end{align*} 
iv) Follows from the definition of $\phi_n^{(l)}$ and the fact that $\psi|_{U_{n+1}U_{n+2}}$ is stabilized by $P_n$.
\end{proof}
\nid 
By iii), $\theta_{n+2}^\chi$ is nontrivial.
By Theorem \ref{thm:1dimprtau}, 
\begin{align}
\Hom_{P_n}(\Ib_\tau(\chi)_n, \Ib_{\tau_n}(\nu^{-1}\chi)) = \C \theta_n^\chi \label{eq:thnchi}
\end{align} 
for $\tau \in \Ir(P_{n-2})$ having a Shalika model relevant to $\chi$.
Since $\theta_r^\chi$ is an integral over a subset of $P_r$, it is also applied to functions $\Phi$ on $P_n, n > r$ such that 
\begin{align*}
\Phi(s p) = \chi_\psi(s) \Phi(p), \ \ s \in \ac{S}_r^\circ.
\end{align*}   
For $\xi \in \psi_n$, set   
\begin{align*}
\Lam_n^\chi \xi := \theta_{n}^{\nu \chi} \circ \cdots \circ \theta_4^{\nu^{m-1} \chi}( \xi) \in \Ib_{\psi_n}(\chi).
\end{align*} 
Denote by $\Lam_n^\chi$ the linear map $\xi \mapsto \Lam_n^\chi \xi$.
If $\Lam_n^\chi$ is nontrivial, then it follows from the irreducibilities of $\psi_n, \Ib_{\psi_n}(\chi)$ 
\begin{thm}\label{thm:PreSLam}
For an arbitrary $\chi$, $\Lam_n^\chi$ is bijective, and 
\begin{align*}
\Hom_{P_n}(\psi_n, \Ib_{\psi_n}(\chi)) = \C \Lam_n^\chi.
\end{align*}
\end{thm}
\nid 
Immediately, follows 
\begin{Cor}\label{Cor:PreSLam}
For arbitrary $\chi$ and $\chi'$, 
\begin{align*}
\Lam_n^{\chi'} \circ (\Lam_n^\chi)^{-1}(J_{n,\chi}^\circ) &= J_{n,\chi'}^\circ, \\ 
\Hom_{P_n}( \Ib_{\psi_n}(\chi), \Ib_{\psi_n}(\chi')) &= \C \Lam_n^{\chi'} \circ (\Lam_n^\chi)^{-1}.
\end{align*}
\end{Cor}
To see the nontriviality of $\Lam_n^\chi$, we observe a particular pre-Shalika form constructed by it as follows.
Fix $\psi$ such that (\ref{ass:psi}).
For $f \in (f_1, \ldots, f_{n-1}) \in \Z^{n-1}$, let 
\begin{align*}
\Vp^f = \diag(\Vp^{f_1}, \ldots, \Vp^{f_{n-1}}).
\end{align*}
Define the $P_n(\oo)$-invariant $\xi_n \in \psi_n$ by
\begin{align}
\xi_n(\ac{\Vp}^f) = 
\begin{cases}
1 & \mbox{if $f = 0 \in \Z^{n-1}$} \\
0 & \mbox{otherwise.} 
\end{cases} \label{def:xin}
\end{align} 
Observe that 
\begin{align*}
\xi_n = \phi_{n}^{(0)} \circ \cdots \circ \phi_2^{(0)}(\xi_2).
\end{align*} 
For an even integer $n \ge 4$, we call  
\begin{align}
J_{n,\chi}^\circ := \theta_n^{\nu \chi} \circ \phi_{n-2}^{(0)} \circ \cdots \circ \theta_4^{\nu^{m-1} \chi} \circ \phi_2^{(0)}(\xi_2) \in \Ib_{\psi_n}(\chi), \label{def:Theta}
\end{align} 
 {\it the essential pre-Shalika form relevant to $\chi$}.
 \begin{prop}\label{prop:Presheq}
With the above notation, 
\begin{align*}
 J_{n,\chi}^\circ = \Lam_n^\chi \xi_n.
 \end{align*} 
\end{prop}
\begin{proof}
Since $\xi_4 = \phi_2^{(0)}(\xi_2)$, the assertion is obvious if $n = 4$.
Suppose that $n \ge 6$.
By the decomposition $P_{n} = \ac{B}_{2n-1} P_n(\oo)$, it suffices to show that
\begin{align}
\theta_{2r-2} \circ \cdots \circ \theta_4 \xi_n(\ac{b}) = \phi_{2r-2} \circ \theta_{2r-2} \circ \cdots \circ \theta_4 \xi_4(\ac{b}), \ \ b \in B_{2r-1}, \label{eq:JLxi}
\end{align} 
for $3 \le r \le n/2$, where we drop the subscripts $\nu^*\chi$ and $0$ from the notation.
We will show this by induction on $r$.
Suppose that $r =3$.
Any matrix in $B_5$ is in the form of 
\begin{align}
u p \begin{bmatrix}
1_{2r-3} &  &  \\
& x &  \\
& & y
\end{bmatrix}, \ p \in P_{2r-2}, x,y \in F^\tm, u \in U_{2r-2} U_{2r-1} \label{eq:B2r+1}
\end{align}
with $r = 3$.
By iv) of Prop. \ref{prop:welldefJW}, at the embedding of this matrix, both sides of (\ref{eq:JLxi}) takes $\psi(u)\theta_4 \xi_4(\ac{h}_{3}p)$, if $x,y \in \oo^\tm$, and zero otherwise.
Here $h_3$ indicates the matrix corresponding to (\ref{def:hn+1}).
Therefore, the statement for $r =3$ is true.
Assume the statement for a particular $r = k$ is true.
Then,  
\begin{align*}
\theta_{2k} \circ \cdots \circ \theta_4 \xi_n(\ac{b}) = \theta_{2k} \circ \phi_{2k-2} \circ \theta_{2k-2} \circ \cdots \circ \theta_4 \xi_4(\ac{b}), \ b \in B_{2k-1}.
\end{align*} 
By the definition of $\phi_{2k}$,  
\begin{align*}
\theta_{2k} \circ \cdots \circ \theta_4 \xi_n(\ac{b}) = \phi_{2k} \circ \theta_{2k} \circ \cdots \circ \theta_4 \xi_4(\ac{b}), \ b \in B_{2k-1}.
\end{align*} 
By the same argument for $r  =3$, this identity holds also for any matrix in $B_{2k+1}$, which in the form of (\ref{eq:B2r+1}) with $r = k+1$.
This establishes the induction step.
\end{proof}
To compute the values of $J_{n,\chi}^\circ$, we need transform the unipotent matrices
\begin{align*}
u(x) &= u_x = \begin{bmatrix}
1_{r} & {}^t x  \\
&1
\end{bmatrix} \in U_{r+1} \\ 
\bu(x) &= \bu_x = \begin{bmatrix}
1_{r} &   \\
x &1
\end{bmatrix} \in \bar{U}_{r+1} 
\end{align*} 
for $x = (x_1, \ldots, x_{r}) \in F^{r}$.
For $a \in F$, let
\begin{align*}
a^* &= \begin{cases}
a^{-1} & \mbox{if $a \neq 0$} \\
0 & \mbox{if $a = 0$} 
\end{cases} \\
\hat{a} &= \begin{cases}
a & \mbox{if $a \neq 0$} \\
1 & \mbox{if $a = 0$.}
\end{cases}
\end{align*} 
Write $x^*$ for $(x_1^*, \ldots, x_r^*)$.
Let $\s = \s_x \in \Sf_{r}$ be the permutation defined by the rules: 
\begin{itemize}
\item $o(x_{\s(r)}) \le \cdots \le o(x_{\s(1)})$.
\item $\s(i) < \s(i+1)$ if $o(x_{\s(i+1)}) = o(x_{\s(i)})$.
\end{itemize}
For $f \in \Z^{r}$, let $\tau_x^f \subset \{1, \ldots, r \}$ be the subset consisting of $i$ such that
\begin{itemize}
\item $x_{\s(i)} \not\in \oo$.
\item $o(x_{\s(i)}) -o(x_{\s(j)}) < f_{\s(j)} - f_{\s(i)} $ for all $j > i$.
\end{itemize}
By definition, 
\begin{align}
 x \not\in \oo^r \ (\Longleftrightarrow) \ x_{\s(r)} \not\in \oo \ (\Longrightarrow) \ \tau_x^f \neq \emptyset. \label{eqv:xo}
\end{align} 
In particular, if $x \not\in \oo^r$, then $r \in \tau_x^f$.
For a subset $\tau$ of $\{1, \ldots, r \}$, let
\begin{align*}
x(\tau) = (x(\tau)_i) \in F^{r}, \ \mbox{where} \ \ 
x(\tau)_i = 
\begin{cases}
x_i & \mbox{if $i \in \s(\tau)$} \\
0 & \mbox{otherwise}.
\end{cases}
\end{align*} 
Let 
\begin{align}
d(x) &= d_x = \diag(d_{x,1}, \ldots, d_{x,r+1}) \in D_{r+1}\notag \\
d_{x,i} &= \begin{cases}
1/\hat{x}_{i} & \mbox{if $i = \s(1)$} \\
\hat{x}_{\s(r)}& \mbox{if $i = r+1$} \\
\hat{x}_{\s(\s^{-1}(i)-1)}/\hat{x}_{i} & \mbox{otherwise}.
\end{cases}
\label{def:dxi}
\end{align}
Let
\begin{align*}
K_r^1 &= \SL_r(\oo) \\
H_f &= (K_r^1)^{\Vp^f}.
\end{align*}
\begin{prop}\label{prop:LST}
With the above notation, $\bu_x$ lies in the double coset 
\begin{align*}
\ac{H_f} u(x(\tau_x^f)^*) d(x(\tau_x^f)) K^1_{r+1}.
\end{align*}  
\end{prop}
\begin{proof}
We prove by induction on $r$.
For $r = 1$, the statement follows immediately from the the Gauss decomposition: 
\begin{align}
\begin{bmatrix}
1&  \\
y &1
\end{bmatrix} = \begin{bmatrix}
1& y^{-1}  \\
&1
\end{bmatrix}
\begin{bmatrix}
& y^{-1} \\
-y &
\end{bmatrix}\begin{bmatrix}
1&y^{-1}  \\
&1
\end{bmatrix}, \ \ y \neq 0. \label{eq:GD}
\end{align} 
Assume the induction hypothesis that the statement for $r = i$ is true.
Let $x \in F^{i+1}$, and $f \in \Z^{i+1}$.
The statement for $i+1$ is obvious if $x$ lies in $\oo^{i+1}$.
Assume $x \not\in \oo^{i+1}$.
Conjugating $\bu_x$ by $\ac{\s_x} \in G_{i+2}$, we may assume that $\s_x$ is the identity, i.e., $o(x_1) \ge \cdots \ge o(x_{i+1})$.
Write
\begin{align*}
x' = (x_{1}, \ldots, x_i), \ f' = (f_{1}, \ldots, f_i), \ y = x_{i+1}.
\end{align*}
Using (\ref{eq:GD}), we transform
\begin{align*}
\bar{u}_x &= 
\begin{bmatrix}
1_{i} & & \\
 &1 &  \\
x' & &1  \\
\end{bmatrix}
\begin{bmatrix}
1_{i} & & \\
 &1 &  \\
 & y &1  \\
\end{bmatrix} \\
&= 
\begin{bmatrix}
1_{i} & & \\
 &1 &  \\
x' & &1  \\
\end{bmatrix}
\begin{bmatrix}
1_{i} & & \\
 &1 & y^{-1}  \\
 & &1  \\
\end{bmatrix} 
\begin{bmatrix}
1_{i} & & \\
 & & y^{-1} \\
 & -y & \\
\end{bmatrix}
\begin{bmatrix}
1_{i} & & \\
 &1 & y^{-1}  \\
 & &1  \\
\end{bmatrix}.
\end{align*} 
Since $y^{-1} \in \oo$, $\bu_x$ lies in the left coset
\begin{align*}
& \begin{bmatrix}
1_{i} & & \\
 &1 &  \\
x' & &1  \\
\end{bmatrix}
\begin{bmatrix}
1_{i} & & \\
 &1 & y^{-1}  \\
 & &1  \\
\end{bmatrix} 
\begin{bmatrix}
1_{i} & & \\
 &y^{-1} &  \\
 & &y  \\
\end{bmatrix}K_{i+2}^1 \\
&= 
\begin{bmatrix}
1_{i} & & \\
 &1 & y^{-1}  \\
 & &1  \\
\end{bmatrix}
\begin{bmatrix}
1_{i} & & \\
 &y^{-1} &  \\
 & &y  \\
\end{bmatrix}
\begin{bmatrix}
1_{i} & & \\
 -x' &1 &  \\
& &1  \\
\end{bmatrix} 
\begin{bmatrix}
1_{i} & & \\
 &1 &  \\
y^{-1}x' & &1 \\
\end{bmatrix} K_{i+2}^1.
\end{align*}
Since $y^{-1}x'$ lies in $\oo^{i+1}$, this equals
\begin{align*} 
\begin{bmatrix}
1_{i} & & \\
 &1 & y^{-1}  \\
 & &1  \\
\end{bmatrix}
\begin{bmatrix}
1_{i} & & \\
 &y^{-1} &  \\
 & &y  \\
\end{bmatrix}
\begin{bmatrix}
1_{i} & & \\
 -x' &1 &  \\
& &1  \\
\end{bmatrix} K_{i+2}^1.
\end{align*} 
Put
\begin{align*}
\tau'' = \tau_{x'}^{f'} \cap \{j \in \{1, \ldots, i \} \mid o(x_j) -o(y) < f_{i+1} -f_j \}, 
\end{align*} 
so that 
\begin{align}
\tau_x^f = \tau'' \sqcup \{i+1 \} = \tau_{x''}^{f'} \sqcup \{i+1 \}. \label{eq:tauxf}
\end{align} 
The last left coset is contained in the double coset
\begin{align*}
H_f \begin{bmatrix}
1_{i} & & \\
 &1 & y^{-1}  \\
 & &1  \\
\end{bmatrix}
\begin{bmatrix}
1_{i} & & \\
 &y^{-1} &  \\
 & &y  \\
\end{bmatrix}\begin{bmatrix}
1_{i} & & \\
x'' &1 &  \\
& &1  \\
\end{bmatrix} K_{i+2}^1, 
\end{align*} 
where we write $x''$ for $x'(\tau'')$.
By the induction hypothesis, 
\begin{align*}
\bu(x'') \in H_{f'} u(x''(\tau_{x''}^{f'})^*)d(x''(\tau_{x''}^{f'})) K_{i+1}^1.
\end{align*} 
Since $\ac{H}_{f'}$ is contained in $H_f$, the last double coset is contained in 
\begin{align*}
& H_f \begin{bmatrix}
1_{i} & & \\
 &1 & y^{-1}  \\
 & &1  \\
\end{bmatrix}
\begin{bmatrix}
1_{i} & & \\
 &y^{-1} &  \\
 & &y  \\
\end{bmatrix} \ac{u}(x''(\tau_{x''}^{f'})^*)\ac{d}(x''(\tau_{x''}^{f'})) K_{i+2}^1 \\
=& 
H_f \ac{u}(y^{-1}x''(\tau_{x''}^{f'})^*) u(x(\tau_{x''}^{f'} \sqcup \{ i + 1\})^*) \begin{bmatrix}
1_{i} & & \\
 &y^{-1} &  \\
 & &y  \\
\end{bmatrix} \ac{d}(x''(\tau_{x''}^{f'})) K_{i+2}^1.
\end{align*} 
By (\ref{eq:tauxf}), this equals 
\begin{align*}
H_f \ac{u}(y^{-1}x''(\tau_{x''}^{f'})^*) u(x(\tau_x^f)^*) d(x(\tau_x^f)) K_{i+2}^1.
\end{align*}  
By (\ref{eq:tauxf}) and the definition of $\tau_x^f$,
\begin{align*}
o(x''(\tau_{x''}^{f'})_j) -o(y) = o(x(\tau_x^f)_j) -o(y) < f_{i+1} -  f_j
\end{align*} 
for all $ j \in \{1, \ldots, i \}$ such that $x_j \not\in \oo$.
By the lemma below, $u(y x''(\tau_{x''}^{f'})^*)$ lies in $H_{f'}$.
Thus the last double coset equals 
\begin{align*}
H_f u(x(\tau_x^f)^*) d(x(\tau_x^f)) K_{i+2}^1.
\end{align*}
This establishes the induction step.
\end{proof}
\begin{rem}\label{rem:ux}
Equivalently, $u_x$ lies in 
\begin{align*}
\ac{H}_{-f} \bu(x(\tau_x^f)^*) d(x(\tau_x^f))^{-1} K^1_{r+1}.
\end{align*} 
\end{rem}
\begin{lem}
Let $f \in \Z^r$.
Let $x \in F^{r-1}$ such that $\s_x$ is the identity.
Then, $u(x(\tau)^*)$ lies in $H_f$, where 
\begin{align*}
\tau = \{i \in \{1, \ldots, r-1\} \mid o(x_{i}) < f_{r} - f_{i} \}.
\end{align*}  
\end{lem}
\begin{proof}
Obvious.
\end{proof}
Now let $n = 2m$ be an even positive integer.
We will abbreviate
\begin{align*}
\chi = \chi \circ \det, \ J_n' = \phi_n^{(0)} J_{n,\chi}^\circ.
\end{align*}
By definition, 
\begin{align*}
J_{n+2,\nu^{-1}\chi}^\circ(p) &= \int_{F^m} \int_{P_m \bs G_{m}} \chi(A)^{-1} J_{n}'(\begin{bmatrix}
A & & &  \\
&A & & \\
& x & 1 &  \\
& & &1  \\
\end{bmatrix}p^\eta) \dr A \dr x.
\end{align*} 
\begin{lem}\label{lem:RSSP}
Any $(S_{n+2}^\circ, P_{n+2}(\oo))$-double coset in $P_{n+2}$ contains a matrix in the form of the embedding into $G_{n+2}$ of   
\begin{align}
\begin{bmatrix}
1_{m} & {}^t \beta\\
& 1 
\end{bmatrix}
\begin{bmatrix}
\Vp^f &  \\
& \Vp^l 
\end{bmatrix}, \ \ f \in \Z^{m}, l \in \Z, \beta \in F^m. \label{def:g_n+2}
\end{align} 
\end{lem}
\begin{proof}
Since $P_{n+2}$ is isomorphic to $G_{n+1} \ltimes U_{n+2}$, any left coset of $P_{n+2}(\oo)$ intersects with $B_{n+2}$ by the Iwasawa decomposition of $G_{n+1}$.
Therefore, any $(S_{n+2}^\circ, P_{n+2}(\oo))$-double coset intersects with $\ac{B}_{m+1}$.
Now, the assertion follows from the Cartan decomposition of $G_{m}$.
\end{proof}
\nid 
Let $g$ be a matrix in the form of (\ref{def:g_n+2}).
Then,
\begin{align*}
J_{n+2, \nu^{-1}\chi}^\circ(g) = 
\int_{F^m} \psi(- {\rm tr}( x {}^t\beta) ) \If_x \dr x,
\end{align*} 
where 
\begin{align}
\If_x := \int_{P_m \bs G_{m}} \chi(A)^{-1} J_n'(\begin{bmatrix}
A \Vp^f & & &  \\
&A & & \\
& & \Vp^l &  \\
& & &1  \\
\end{bmatrix}
\begin{bmatrix}
1_m & & &  \\
& 1_m & & \\
& \Vp^{-l} x & 1&  \\
& & &1  \\
\end{bmatrix}) \dr A. \label{def:Ix}
\end{align} 
Write $y$ for $\Vp^{-l} x$.

First, consider the case where $\chi$ is ramified, i.e., $e > 0$.
If $y$ lies in $\oo^{m}$, then $\If_x$ equals
\begin{align*}
\int_{P_m \bs G_{m}} \chi^{-1}(A) J_n'(\begin{bmatrix}
A \Vp^f  & & &  \\
&A & & \\
& & \Vp^l &  \\
& & &1  \\
\end{bmatrix}) \dr A.
\end{align*} 
This is zero since $\chi$ is ramified.
Assume that 
\begin{align*}
y \not\in \oo^m.
\end{align*}
Abbreviate 
\begin{align*}
\s = \s_y, \ \tau = \tau_y^f, \ y_0 = y_{\s(m)}.
\end{align*}
Since $y \not\in \oo^m$, $\hat{y}_{0}$ equals $y_{0}$.
By Prop. \ref{prop:LST}, there exists an $h \in H_f$ such that $\bu_y \in \ac{h} u_{y(\tau)^*}d_{y(\tau)} K^1_{m+1}$.
Write 
\begin{align*}
d_{y(\tau)} = \begin{bmatrix}
d^\circ &  \\
& y_{0}
\end{bmatrix}.
\end{align*} 
Define $\dot{y} \in F^{m}$ by 
\begin{align*}
u(y(\tau)^*)d(y(\tau)) = d(y(\tau)) u(\dot{y}).
\end{align*}  
Then, 
\begin{align*}
\If_x &= \int \chi(A)^{-1} J_n'(\begin{bmatrix}
A \Vp^f & & &  \\
&A h d^\circ & & \\
& &\Vp^l y_{0}&  \\
& & &1  \\
\end{bmatrix}
\begin{bmatrix}
1_m & & &  \\
& 1_m & {}^t \dot{y} & \\
& & 1&  \\
& & &1  \\
\end{bmatrix}) \dr A \\
=& \chi(y_0)^{-1}\int \chi(A)^{-1} J_n'(\begin{bmatrix}
A d_{}^{\circ -1} h^{-1} \Vp^f & & &  \\
&A & & \\
& & \Vp^l y_{0} &  \\
& & &1  \\
\end{bmatrix}
\begin{bmatrix}
1_m & & &  \\
& 1_m & {}^t \dot{y}& \\
& & 1&  \\
& & &1  \\
\end{bmatrix}) \dr A.
\end{align*} 
where the integrals in $A$ are over $P_m \bs G_{m}$. 
The integrated function equals 
\begin{align}
\chi(A)^{-1} \psi\left(\frac{(A {}^t \dot{y})_m}{\Vp^l y_{0}}\right) J_n'(\begin{bmatrix}
A d_{}^{\circ -1}\Vp^f & & &  \\
&A & & \\
& & \Vp^l y_{0} &  \\
& & &1  \\
\end{bmatrix}). \label{eq:intIxfn}
\end{align}  
Define diagonal matrices $\delta_m, \delta_m^\sharp \in D_m$ by
\begin{align}
\begin{split}
\delta_m &= 
\begin{cases}
1  & \mbox{if $m=1$} \\
\diag(\Vp^e, \Vp^{3e}, \ldots, \Vp^{(2m-3)e}, \Vp^{(m-1)e}) & \mbox{if $m > 1$} 
\end{cases}\\
\delta_m^\sharp &= 
\begin{cases}
1 & \mbox{if $m = 1$} \\
\diag(1, \Vp^{2e}, \ldots, \Vp^{2(m-1)e} ) & \mbox{if $m > 1$.} 
\end{cases}
\end{split}
\label{def:dm}
\end{align}
Define unipotent matrices $v_m  \in U_m, v_m^\sharp  \in \bar{U}_m$ by 
\begin{align}
\begin{split}
v_m &= u((\Vp^{-(m-1)e}, \Vp^{-(m-2)e }, \ldots, \Vp^{-e})) \\
v_m^\sharp&= \bu((\Vp^{(m-1)e},\Vp^{(m-2)e }, \ldots, \Vp^e)).
\end{split}
\label{def:um}
\end{align} 
One can see by using Prop. \ref{prop:LST} that
\begin{align*}
S_{n}^\circ \ac{v}_{m} \ac{\delta}_{m} P_{n}(\oo) = S_{n}^\circ \ac{v}^\sharp_{m} \ac{\delta}_{m}^\sharp P_{n}(\oo).
\end{align*} 
\begin{thm}\label{thm:suppJ}
If $\chi$ is ramified, then the support of $J_{n,\chi}^\circ$ is the above double coset space.
\end{thm}
\nid
We prove this by induction.
The statement for $m =  1$ is obviously true.
Assuming the induction hypothesis that the statement for a particular $m$ is true, we will prove the statement for $m+1$ is true by computing $J_{n+2, \nu^{-1}\chi}^\circ(g)$.
For $w \in \Sf_m$, and $t = (t_1, \ldots, t_m), t_i \in \oo^\tm$, set
\begin{align}
A_w(t) = \begin{bmatrix}
1& & &  \\
& \ddots & & \\
& & 1&  \\
\Vp^{w_-(1)e}t_{1} & \cdots & \Vp^{w_-(m-1)e}t_{m-1} & \Vp^{w_-(m)e}t_{m}  \\
\end{bmatrix} \label{def:Awt}
\end{align} 
where $w_-: \{1, \ldots, m\} \to \{0, 1, \ldots,m-1\}$ is the map:  
\begin{align*}
j \longmapsto m-w(j).
\end{align*} 
\begin{prop}\label{prop:Aform}
There exists a unique $w \in \Sf_m$ such that $d_{}^{\circ -1} \Vp^{f} \in (\delta_m^\sharp)^w K_m$,  and $A \in P_m A_w(t)$ for some $t = (t_1, \ldots, t_m), t_i \in \oo^\tm$, if and only if the function (\ref{eq:intIxfn}) does not vanish at $A \in P_m \bs G_m$.
\end{prop}
\begin{proof}
The coset space $P_m \bs G_m$ is realized as the set of the following matrices:
\begin{align}
& \begin{bmatrix}
1_{m-1}&  \\
* & *
\end{bmatrix}, \label{pgtype1} \\
& \begin{bmatrix}
1_j & &  \\
& 1_{m-1-j}&  \\
* & & *
\end{bmatrix} s_j, \ j \in \{0, \ldots, m -2\}, \label{pgtype2}
\end{align}
where 
\begin{align}
s_j = \left[
\begin{array}{cccc}
1_{j} &   & & \\ 
   & 0 &  &1 \\
    & &1_{m-j-2} &  \\ 
  & 1  & & 0
\end{array}
\right]. \label{def:sj}
\end{align}  
Write $\{ \mu \}$ for this system.
Abbreviate 
\begin{align*}
a = d_{}^{\circ -1} \Vp^{f}, \ \bu = v_m^\sharp.
\end{align*}
Assume that the function (\ref{eq:intIxfn}) does not vanish at $A = \mu$.
By the definition of $J_n'$ and the induction hypothesis for $m$, there exists a $p \in P_{m}$ such that $\mu \in p K_m, \mu a \in p \bu \delta_m^\sharp  K_m$, equivalently, 
\begin{align}
p^{-1} \mu, (\delta_m^\sharp)^{-1} \bu^{-1} p^{-1} \mu a \in K_m. \label{eq:propp''v}
\end{align} 
In particular, we need the condition $a \in K_m \delta_m^\sharp K_m$.
By the Cartan decomposition, there exists a unique $w \in \Sf_m$ such that
\begin{align*}
a \in (\delta_m^\sharp)^w K_m.
\end{align*} 
Assume that $\mu$ is in the form of (\ref{pgtype2}).
Let $\mu' = \mu s_j^{-1}$.
Then, (\ref{eq:propp''v}) is equivalent to 
\begin{align*}
p^{-1} \mu', (\delta_m^\sharp)^{-1} \bu^{-1} p^{-1} \mu' (\delta_m^\sharp)^{s_j w} \in K_m.
\end{align*} 
Therefore, we need
\begin{align*}
\bu^{-1}  p^{-1} \mu' \in \delta_m^\sharp K_m ((\delta_m^\sharp)^{s_jw})^{-1}.
\end{align*}
But, by the lemma below, 
\begin{align*}
(\bu^{-1}  p^{-1} \mu')_{i,m-1} 
\in \begin{cases}
\oo & \mbox{if $i < s_jw(m-1)$} \\
\oo^\tm  & \mbox{if $i = s_jw(m-1)$} \\
\p^{2e(i- s_jw(m-1))} & \mbox{if $i > s_jw(m-1)$},
\end{cases}
\end{align*} 
and therefore, 
\begin{align*}
o((p^{-1} \mu')_{m,m-1}) &= o((\bu \bu^{-1} p^{-1} \mu')_{m, m-1})\\
 &= 2e(m- s_j w(m-1)).
\end{align*}
This is a contradiction since $(p^{-1} \mu')_{m,m-1}$ is zero.
Hence, we may assume that $\mu$ is in the form of (\ref{pgtype1}).
Observing the last raw of $p^{-1} \mu = \bu \bu^{-1} p^{-1} \mu$ with using the same lemma and the condition $\bu^{-1} p^{-1} \mu \in \delta_m^\sharp K_m ((\delta_m^\sharp)^{w})^{-1}$, we conclude that, for some $t = (t_1, \ldots, t_m), t_i \in \oo^\tm$, $\mu$ equals $A_w(t)$, i.e., 
\begin{align*}
 A \in P_mA_w(t).
 \end{align*}
The converse direction is obvious. 
\end{proof}
\begin{lem}
Let $f \in \Z^r$ with $f_1 < \cdots < f_r$.
If $k$ lies in $K_r \cap \Vp^f K_r ((\Vp^f)^w)^{-1}$ for $w \in \Sf_r$, then it holds that
\begin{align*}
k_{i,w^{-1}(i)} \in \oo^\tm, \ 
k_{i,j} \in \begin{cases}
\p^{f_i - f_{w(j)}} & \mbox{if $f_{w(j)} < f_i$} \\
\oo & \mbox{otherwise}. 
\end{cases}
\end{align*}
\end{lem}
\begin{proof}
By an elementary argument.
\end{proof}
\nid
Put
\begin{align*}
\Jf_{w}(t) = \chi(A_w(t))^{-1}J_n'(\begin{bmatrix}
A_{w}(t) d_{}^{\circ -1}\Vp^f & & &  \\
&A_{w}(t) & & \\
& &  \Vp^l y_{0} &  \\
& & &1  \\
\end{bmatrix}).
\end{align*} 
Abbreviate $\Jf_w((1,\ldots,1))$ as $\Jf_w$.
Observe that
\begin{align}
\Jf_w((t_1,\ldots, t_m)) = \Jf_w \prod_{j=1}^{m} \chi^{-1}(t_j). \label{eq:propJCw}
\end{align} 
Since we assumed that $\Jf_{w}$ is nonzero, by Prop. \ref{prop:Aform} 
\begin{align}
l = - o(y_{0}). \label{eq:l-y0}
\end{align}
Using (\ref{eq:propJCw}), we compute 
\begin{align*}
\If_x &= \chi(y_0)^{-1} \sum_{w \in \Sf_{m}} \Jf_{w} \int_{(\oo^\tm)^m} \prod_{j=1}^{m}  \chi(t_j)^{-1} \psi\left(\sum_{j= 1}^{m} \frac{\dot{y}_{j}\Vp^{w_-(j)e} t_j}{y_{0}\Vp^l} \right) \dr t_1 \cdots \dr t_m \\
&= \chi(y_0)^{-1} \sum_{w \in \Sf_{m}} \Jf_{w} \prod_{j=1}^{m} \gf\left(\chi^{-1}, \psi_{\dot{y}_{j}\Vp^{w_-(j)e-l}/y_0}\right).
\end{align*}
Here $\psi_a$ for $a\in F^\tm$ indicates the additive character $x \mapsto \psi(ax)$, and $\gf(\chi^{-1},\psi)$ the Gauss sum defined by
\begin{align}
\gf(\chi^{-1},\psi) = \int_{\oo} \chi^{-1}(x) \psi(x) \dr x, \label{def:Gaussum}
\end{align} 
where we extend $\chi$ to $F$ by $\chi(0) = 0$, and the Haar measure is normalized so that $\vol(\oo) = 1$.
It is well-known that, for $\psi$ satisfying (\ref{ass:psi}) and ramified $\chi$, 
\begin{align}
\gf(\chi,\psi_a) \neq 0 \ (\Longleftrightarrow) \ o(a) = -e. \label{eq:GSsup}
\end{align} 
Therefore, the last product of Gauss sums is nonzero, if and only if   
\begin{align} 
o \left(\dot{y}(\tau)_{j}\right) = -(w_-(j)+1)e, \ \ j \in \{1, \ldots, m\}. \label{eq:oydot}
\end{align}
This implies that $y$ lies in $(F \setminus \oo)^{m}$.
Thus $\tau = \tau_y^f$ is the full-set, and $y(\tau)$ equals $y$.
By the definition of $\dot{y}$,
\begin{align}
\dot{y}_j = 
\begin{cases}
y_{0}& \mbox{if $j = \s(1)$} \\
y_{0}/y_{\s(\s^{-1}(j)-1)}& \mbox{otherwise.} 
\end{cases}
\label{eq:doty} 
\end{align} 
Combining with (\ref{eq:oydot}), we obtain an identity
\begin{align*}
\{ o(y_{0})\} \sqcup \{o(y_{0}) -o(y_{j}) \}_{j \neq 0, \s(m)} = \{ -e, \cdots, - m e\} 
\end{align*} 
as sets.
It follows that 
\begin{align*}
o(y_{\s(j)}) = -je, \ j \in \{1, \ldots, m\}.
\end{align*} 
Therefore,  
\begin{align}
\s^{-1}(j) + w_-(j) &= m, \ j \in \{1,\ldots, m\} \label{eq:swrel}
\end{align} 
by (\ref{eq:oydot}), and 
\begin{align*}
l = me
\end{align*} 
by (\ref{eq:l-y0}).
From the definition of $d(y)$, it follows that 
\begin{align*}
o(d_{j,j}^\circ) = e, \ j \in \{1,\ldots, m\}
\end{align*} 
Therefore, from Prop. \ref{prop:Aform}, it follows that  
\begin{align*}
f_j = (2w(j)-1)e, \ j \in \{1,\ldots, m\}.
\end{align*} 
This means that $\Jf_{w}$ for only one $w$ as above is nonzero, and other $\Jf_{w'}$ are zero.
Now, it follows from (\ref{eq:doty}), (\ref{eq:swrel}) that 
\begin{align*} 
\If_x &= \Jf_{w} \chi^{-1}(y_{0}) \gf(\chi^{-1}, \psi_{\Vp^{-e}}) \prod_{j=1}^{m-1} \gf\left(\chi^{-1}, \psi_{\frac{\Vp^{(m-j-1)e-l}}{y_{\s(j)}}}\right) \\
&= \Jf_{w} \chi^{-1}(y_{0}) \gf(\chi^{-1}, \psi_{\Vp^{-e}}) \prod_{j=1}^{m-1} \gf\left(\chi^{-1}, \psi_{\frac{\Vp^{(m-j-1)e}}{x_{\s(j)}}}\right).
\end{align*} 
We find that
\begin{align*}
\supp(\If_x) = \Vp^{(m-\s^{-1}(1))e} \oo^\tm \tm \cdots \tm \Vp^{(m-\s^{-1}(m))e} \oo^\tm,
\end{align*} 
and can express 
\begin{align*}
\If_x = c \chi(\prod_{j=1}^m x_j)^{-1}, c \neq 0 
\end{align*} 
by the identity
\begin{align}
\gf(\chi,\psi_{ab}) = \nu\chi(b)^{-1} \gf(\chi,\psi_{a}) ,a,b \in F^\tm. \label{eq:GSprop}
\end{align}
Therefore,  
\begin{align*}
& \int_{F^{m}} \psi(x {}^t\beta) \If_x \dr x = \int_{\supp(\If_x)} \psi(x {}^t\beta) \If_x \dr x \\
&= c\prod_{j=1}^m \int_{\Vp^{(m-\s^{-1}(j))e} \oo^\tm} \chi(x_j)^{-1} \psi(\beta_j x_j) \dr x_j \\
&= c \prod_{j =1}^m \nu\chi(\Vp)^{(m-\s^{-1}(j))e} \gf\left(\chi^{-1}, \psi_{\beta_j \Vp^{(m -\s^{-1}(j))e}}\right).
\end{align*} 
By (\ref{eq:GSsup}), the Gauss sums are nonzero, if and only if 
\begin{align*}
o(\beta_j) = (\s^{-1}(j) -m-1)e, \ \ j \in \{1,\ldots, m\}.
\end{align*} 
This condition is equivalent to 
\begin{align*}
o(\beta_j) = (w(j) -m-1)e, \ \ j \in \{1,\ldots, m\}
\end{align*} 
by (\ref{eq:swrel}).
This establishes the induction step, completing the proof for the theorem.

Next, consider the case where $\chi$ is unramified.
Let 
\begin{align*}
P_n^*&= \{p \in P_n \mid \det(p) \in \oo^\tm\} \\
P_n^*(\oo)&= P_n^* \cap P_n(\oo) \\
S_n^{\circ *}&= S_n \cap P_n^*.
\end{align*} 
We will prove:
\begin{thm}\label{thm:unrpsi}
If $\chi$ is unramified, then the support of $J_{n,\chi}^\circ|_{P_n^*}$ is $(S_n^{\circ *} \cap P_n^*) P_n^*(\oo)$.
\end{thm}
\nid
The proof of the next lemma is similar to that of Lemma \ref{lem:RSSP}, and omitted.
\begin{lem}
Any $(S_{n}^{\circ*}, P_{n}^*(\oo))$-double coset in $P_{n}^*$ contains a matrix in the form of the embedding into $G_{n}$ of   
\begin{align*}
u_\beta \Vp^f, \ \beta \in F^{m-1}, \ f \in \Z^{m}, \sum_{i=1}^{m} f_i =0.
\end{align*} 
\end{lem}
\nid
For the matrices in this lemma, we have:
\begin{prop}\label{prop:chiurfnonzeo}
Let $J$ be a $P_n(\oo)$-invariant pre-Shalika form.
Take a nonzero $f \in \Z^m$ such that $\sum_{i=1}^m f_i = 0$.
Then, $\ac{J}(u_\beta \Vp^f)$ vanishes for an arbitrary $\beta \in F^m$.
\end{prop}
\begin{proof}
Some $f_i$ is negative by our assumption.
The assertion follows from the next lemma.
\end{proof}
\begin{lem}\label{lem:vanishlem}
Let $\Omega$ be a field.
Let $G$ be a group, and $H, K$ be its subgroups.
Let $\xi$ and $\w$ be homomorphisms into $\Omega^\tm$ of $H$ and $K$, respectively.
Let $J$ be a $\Omega$-valued function on $G$ such that 
\begin{align*}
J(hgk) = \xi(h)\w(k) J(g), \ \ h \in H, g \in G, k \in K.
\end{align*} 
Then $J(g_0)$ vanishes at $g_0 \in G$, if there exists an $h \in H$ such that $h^{g_0} \in K$ and $\xi(h) \neq \w(g_0^{-1} h g_0)$.
\end{lem} 
\begin{proof}
Obvious.
\end{proof}
\nid
Therefore, for the computation of $\supp(J_{n,\chi}^\circ|_{P_n^*})$, it suffices to examine the values at $\ac{u}_\beta$ for $\beta \in F^{m-1}$.
We will prove the theorem by induction on $m$.
The statement for $m =  1$ is obviously true.
Assuming the induction hypothesis that the statement for a particular $m$ is true, we will prove the statement for $m+1$ is true by computing $J_{n+2,\nu^{-1}\chi}^\circ(\ac{u}_\beta)$ for $\beta \in F^m$.
In this situation, $\If_x$ (see (\ref{def:Ix}) for the definition) equals
\begin{align*}
\int_{P_m \bs G_{m}} \chi(A)^{-1} J_n'(\begin{bmatrix}
A & & &  \\
&A & & \\
& x & 1 &  \\
& & &1  \\
\end{bmatrix}) \dr A.
\end{align*} 
Assume that $x \in F^m$ does not lie in $\oo^m$.
Abbreviate $\tau_x^0$ as $\tau$.
By Prop. \ref{prop:LST}, there exist $h \in K^1_m$ and $k \in K^1_{m+1}$ such that $\bu_{x} \in \ac{h} \bu(x(\tau)^*) d(x(\tau))^{-1}k$.
Using this expression, we compute
\begin{align*}
\If_x &= \int \chi(A)^{-1} J_n'(
\begin{bmatrix}
A& &  \\
&A h&  \\
& & 1_2
\end{bmatrix}
\begin{bmatrix}
1& &  \\
&u_{x(\tau)^*} d_{x(\tau)} &  \\
& & 1
\end{bmatrix}) \dr A \\
&= \int \chi(A)^{-1} J_n'(
\begin{bmatrix}
A h^{-1} & &  \\
&A &  \\
& &1_2 
\end{bmatrix}
\begin{bmatrix}
1& &  \\
&u_{x(\tau)^*} d_{x(\tau)} &  \\
& & 1
\end{bmatrix}) \dr A \\
&= \int \chi(A)^{-1} J_n'(
\begin{bmatrix}
A & &  \\
&A &  \\
& &1_2 
\end{bmatrix}
\begin{bmatrix}
1& &  \\
&u_{x(\tau)^*} d_{x(\tau)} &  \\
& & 1
\end{bmatrix}) \dr A
\end{align*} 
where the integrals in $A$ are over $P_m \bs G_{m}$. 
But, $x(\tau)_{\s_{x(\tau)}(n)} = x_{\s_x(n)}$ does not lie in $\oo$ by (\ref{eqv:xo}).
Therefore, the integrated function vanishes by Prop. \ref{prop:welldefJW}.
If $x$ lies in $\oo^m$, then 
\begin{align*}
J_n'(\begin{bmatrix}
A & & &  \\
&A & & \\
& x & 1 &  \\
& & &1  \\
\end{bmatrix}) = 
J_n'(\begin{bmatrix}
A & & &  \\
&A & & \\
& & 1 &  \\
& & &1  \\
\end{bmatrix}).
\end{align*} 
Assume that this is nonzero at $A$.
Then, there exist $\ac{\Vp}^f, f \in \Z^m$ at which $J_n'$ does not vanish, and $p \in P_m, k,k' \in K_m$ such that 
\begin{align*}
A = p \Vp^f k = p k'.
\end{align*} 
Therefore, $o(f_1) = \cdots = o(f_m) = 0$.
By the induction hypothesis, in this situation, the above function takes $J_n'(1_{n+2})$.
By the construction, $J_n'(1_{n+2}) = J_{n,\chi}^\circ(1_n)$.
Therefore, 
\begin{align*}
\If_x = \If_0 = \int_{P_m \bs P_mK_m} J_n'(\begin{bmatrix}
A & & &  \\
&A & & \\
& & 1 &  \\
& & &1  \\
\end{bmatrix}) \dr A = J_{n,\chi}^\circ(1_n).
\end{align*} 
We have 
\begin{align*}
J_{n+2,\nu^{-1} \chi}^\circ(\ac{u}_\beta) 
&= \int_{\oo^m} \psi(x{}^t\beta)\If_0 \dr x \\
&= \begin{cases}
J_{n,\chi}^\circ(1_n) & \mbox{if $\beta$ lies in $\oo^m$} \\
 0& \mbox{otherwise.} 
\end{cases}
\end{align*} 
This establishes the induction step, and completes the proof.

By the argument of \cite{S}, if an unramified irreducible representation $\pi$ of $G_{n}$ has a Shalika model relevant to $\1$,  its unique $K_{n}$-invariant Shalika form does not vanish at the identity.
His argument works also for the case where $\pi$ and $\chi$ are both unramified.
Therefore, by the argument for the above proposition, we obtain:
\begin{prop}
Let $d < n$ be even integers.
If an unramified $\pi \in \Ir(G_{d})$ has a Shalika model relevant to $\chi$ (unramified by Prop. \ref{prop:genPSur}), then, there exists a $P_{n}(\oo)$-invariant vector in $\Ib_{\pi_{n}}(\nu^{(d-n)/2} \chi)$ which does not vanish at the identity.
\end{prop}
\section{Whittaker model}\label{sec:Wm} 
For $J \in \Ib_{\psi_{n}}(\chi)$, define the function $\Xi_{n}J$ on $P_{n}$ by
\begin{align*}
\Xi_{n}J(p) 
&= \int_{F^{m-1}} \int_{F^{m-1}} \psi(z_{m-1}) J(\ac{h}_{n-1}^{-1}\begin{bmatrix}
1_{m-1}&& {}^t y&  \\
& 1_{m-1}& {}^t z & \\
&  & 1&  \\
& & & 1 \\
\end{bmatrix}p) \dr y \dr z \\
&= \int_{F^{m-1}} \int_{F^{m-1}} \psi(z_{m-1}) J(\begin{bmatrix}
1_{m-1}&{}^t y & &  \\
& 1& & \\
& {}^t z & 1_{m-1}&  \\
& & & 1 \\
\end{bmatrix}h_{n-1}^{-1} p) \dr y \dr z. 
\end{align*} 
where $h_{n-1} \in \Sf_{n-1}$ indicates the permutation matrix defined at (\ref{def:hn+1}).
It is not hard to see the convergence.
Denote the mapping $J \mapsto \Xi_{n}J$ by $\Xi_{n}$, which lies in $\Hom_{P_{n}}(\Ib_{\psi_{n}}(\chi), \Ib_{\psi_{n}}(\nu^2\chi)_{n})$.
Let $J_{n+2,\chi}^\circ$ be the essential pre-Shalika form.
Put
\begin{align*}
(\bu, \delta)= 
\begin{cases}
(1_m,1_m) & \mbox{if $\chi$ is unramified} \\
(v_m^\sharp,\delta_m^\sharp) & \mbox{if $\chi$ is ramified} 
\end{cases}
\end{align*}
where $\delta_m^\sharp \in D_m$ and $v_m^\sharp \in \bar{U}_m,$ are the matrices defined at (\ref{def:dm}) and (\ref{def:um}), respectively.
Abbreviate
\begin{align*}
g &= g_{n+2} = \ac{\bu}\ac{\delta} \in P_{n+2} \\
\eta &= \ac{h}_{n+1} \in P_{n+2} \\
J' &= \theta_{n+2}^{\nu \chi} \phi_n^{(0)} J_{n, \nu\chi}^\circ\\
J^\circ &= J_{n+2,\chi}^\circ.
\end{align*}
We will prove $\Xi_{n+2}$ is nontrivial by showing that $\Xi_{n+2} J_{n+2}^\circ(g) \neq 0$.
Since $\eta^{-1}g\eta$ equals $g$, and $J^\circ$ is $P_{n+2}(\oo)$-invariant, we have
\begin{align*}
\Xi_{n+2} J^\circ(g) = \int_{F^{m}} \int_{F^{m}} \psi(z_m) J^\circ(\begin{bmatrix}
\bu \delta &{}^t y & &  \\
& 1& & \\
& {}^t z & 1_{m}&  \\
& & & 1 \\
\end{bmatrix}) \dr y \dr z.
\end{align*}
By the definition of $J^\circ$, this equals
\begin{align} 
\iiiint \frac{\psi(z_m)}{\chi(A)} J'(\eta \begin{bmatrix}
A& & &  \\
& 1 & x & \\
& &A &  \\
& & &1  \\
\end{bmatrix}\begin{bmatrix}
\bu \delta &{}^t y & &  \\
& 1& & \\
& {}^t z & 1_{m}&  \\
& & & 1 \\
\end{bmatrix}) \dr A \dr x \dr y \dr z \label{eq:WMint1st}
\end{align} 
where the integral in $A$ is over $P_m \bs G_m$, and those in $x, y,z$ are over $F^m$.
By a matrix computation, the integrated function equals
\begin{align*}
\frac{\psi(z_m) \psi(x{}^ty)}{\chi(A)} J'(\eta \begin{bmatrix}
A& & &  \\
& 1 & x  & \\
&  &A&  \\
& & &1  \\
\end{bmatrix}\begin{bmatrix}
\bu \delta  && &  \\
& 1& & \\
& {}^t z & 1_{m}&  \\
& & & 1 \\
\end{bmatrix}). 
\end{align*}
Fixing $z$, we compute 
\begin{align*}
& \int_{F^m} \int_{P_m \bs G_m} \frac{\psi(x{}^ty)}{\chi(A)} J'(\eta \begin{bmatrix}
A \bu & & &  \\
& 1 & x  & \\
&  &A&  \\
& & &1  \\
\end{bmatrix}
\begin{bmatrix}
\delta & & &  \\
& 1& & \\
& {}^tz &1_m &  \\
& &  &1  \\
\end{bmatrix}) \dr A \dr x \\
&= 
 \int_{F^m} \int_{P_m \bs G_m}\frac{\psi(x{}^ty)}{\chi(A)} J'(\eta \begin{bmatrix}
A & & &  \\
& 1 & x  & \\
&  &A \bu^{-1} &  \\
& & &1  \\
\end{bmatrix}
\begin{bmatrix}
\delta & & &  \\
& 1& & \\
& {}^tz &1_m &  \\
& &  &1  \\
\end{bmatrix}) \dr A \dr x.
\end{align*}
This equals 
\begin{align*}
\int_{F^m} \int_{P_m \bs G_m} \frac{\psi(x\bu^{-1}{}^ty)}{\chi(A)} J'(\eta \begin{bmatrix}
A & & &  \\
& 1 & x  & \\
&  &A &  \\
& & &1  \\
\end{bmatrix}
\begin{bmatrix}
\delta & & &  \\
& 1& & \\
& \bu^{-1} {}^tz &1_m &  \\
& &  &1  \\
\end{bmatrix}) \dr A \dr x
\end{align*} 
by the $P_{n+2}(\oo)$-invariance property and the matrix computation:
\begin{align*}
& \begin{bmatrix}
A & & &  \\
& 1 & x  & \\
&  &A \bu^{-1} &  \\
& & &1  \\
\end{bmatrix}
\begin{bmatrix}
1_m & & &  \\
& 1& & \\
& &\bu &  \\
& &  &1  \\
\end{bmatrix} = \begin{bmatrix}
A & & &  \\
& 1 & x \bu  & \\
&  &A &  \\
& & &1  \\
\end{bmatrix}, \\
& \begin{bmatrix}
1_m & & &  \\
& 1& & \\
& &\bu^{-1} &  \\
& &  &1  \\
\end{bmatrix} \begin{bmatrix}
\delta & & &  \\
& 1& & \\
& {}^tz &1_m &  \\
& &  &1  \\
\end{bmatrix}
\begin{bmatrix}
1_m & & &  \\
& 1& & \\
& &\bu &  \\
& &  &1  \\
\end{bmatrix} = 
\begin{bmatrix}
\delta & & &  \\
& 1& & \\
& \bu^{-1}{}^tz &1_m &  \\
& &  &1  \\
\end{bmatrix}.
\end{align*} 
Consequently, at (\ref{eq:WMint1st}), we replace the integrated function with 
\begin{align} 
\frac{\psi(\bu_m{}^tz + x{}^ty)}{\chi(A)} J'(\eta \begin{bmatrix}
A & & &  \\
& 1 & x & \\
&  &A  &  \\
& & &1  \\
\end{bmatrix}
\begin{bmatrix}
\delta & & &  \\
& 1& & \\
& {}^tz  &1_m &  \\
& &  &1  \\
\end{bmatrix}), \label{eq:1stthl}
\end{align} 
where $\bu_m$ indicates the last row of $\bu$.
Let
\begin{align*}
\If_{x,z} &= \int_{P_m \bs G_m} \chi(A)^{-1} J'(\eta \begin{bmatrix}
A & & &  \\
& 1 & x & \\
&  &A  &  \\
& & &1  \\
\end{bmatrix}
\begin{bmatrix}
\delta & & &  \\
& 1& & \\
& {}^tz  &1_m &  \\
& &  &1  \\
\end{bmatrix}) \dr A \\
\If_{y,z}' &= \int_{F^m} \psi(x{}^ty) \If_{x,z} \dr x  \\ 
\If_z &= \int_{F^m} \If_{y,z}' \dr y.
\end{align*} 

Suppose that $z \in \oo^m$.
Then, (\ref{eq:1stthl}) equals
\begin{align*}
\frac{\psi(x{}^ty)}{\chi(A)} J'(\begin{bmatrix}
A\delta & & &  \\
& A& & \\
& x & 1  &  \\
& & &1  \\
\end{bmatrix}).
\end{align*} 
If $x \not\in \oo^m$, then by Prop. \ref{prop:LST}, there exist $k \in K_{m+1}^1, \ h \in K_{m}^1, a \in D_{m}, u \in U_m$ such that 
\begin{align*}
\bu_x = \ac{h} u \begin{bmatrix}
a &  \\
& x_{\s_x(m)}
\end{bmatrix}k.
\end{align*}
Using this expression, we compute 
\begin{align*}
J'(\begin{bmatrix}
A\delta & & &  \\
& A& & \\
& x & 1  &  \\
& & &1  \\
\end{bmatrix}) = J'(\begin{bmatrix}
A\delta & & &  \\
& A h a& *& \\
& & x_{\s_x(m)} &  \\
& & &1  \\
\end{bmatrix}).
\end{align*} 
This is zero by Prop. \ref{prop:Aform}, since $x_{\s_x(m)} \not\in \oo$.
Hence, 
\begin{align}
\begin{split}
\If'_{y,z} &=\int_{\oo^m} \psi(x{}^ty) \int_{P_m\bs G_m }\chi(A)^{-1} J'(\begin{bmatrix}
A\delta & & &  \\
& A& & \\
& & 1  &  \\
& & &1  \\
\end{bmatrix}) \dr A \dr x \\
&= 
\begin{cases}
1 & \mbox{if $\chi$ is unramified and $y$ lies in $\oo^m$} \\
0 & \mbox{otherwise.} 
\end{cases}
\end{split}\label{eq:Ifyz}
\end{align} 

Suppose that $z \not\in \oo^m$.
For $d \in D_m$, let $o(d)$ denote $(o(d_{1,1}), \ldots, o(d_{m,m}))$.
Let $\bar{z} = (z_m,z_{m-1}, \ldots, z_1)$ and $\bar{\delta} = \diag(\delta_{m,m}, \ldots, \delta_{1,1})$.
Abbreviate $z_{\s_z(m)}$ and $\tau_{\bar{z}}^{-o(\bar{\delta})}$ as $z_0$ and $\bar{\tau}$ respectively.
By Remark \ref{rem:ux}, there exist $h \in (K_m^1)^{o(\delta)}, d \in D_m$ and $k \in K_{m+1}^1$ such that 
\begin{align*}
\begin{bmatrix}
1&  \\
{}^tz & 1_m
\end{bmatrix} = \begin{bmatrix}
1&  \\
& h
\end{bmatrix}
\begin{bmatrix}
1 & z(\bar{\tau})^*  \\
&1_m
\end{bmatrix}
\begin{bmatrix}
z_{0}^{-1} &  \\
& d
\end{bmatrix}k.
\end{align*} 
Observe that 
\begin{align}
\overline{d(z(\bar{\tau}))} = 
\begin{bmatrix}
z_{0} &  \\
& d^{-1}
\end{bmatrix}.  \label{def:d}
\end{align} 
Using the above expression, we compute 
\begin{align*}
&\If'_{y,z} = \iint \frac{\psi(x{}^ty)}{\chi(A)} J'(\eta \begin{bmatrix}
A & & &  \\
& 1 & x & \\
&  &A  &  \\
& & &1  \\
\end{bmatrix}
\begin{bmatrix}
\delta & & &  \\
& 1& & \\
& {}^tz  &1_m &  \\
& &  &1  \\
\end{bmatrix}) \dr A \dr x \\
&= \iint \frac{\psi(x{}^ty)}{\chi(A)} J'(\eta \begin{bmatrix}
A& & &  \\
& 1 & x h +z(\bar{\tau})^* & \\
&  &A h&  \\
& & &1  \\
\end{bmatrix}
\begin{bmatrix}
\delta & & &  \\
&z_{0}^{-1} & & \\
& & d &  \\
& & &1  \\
\end{bmatrix}) \dr A \dr x  \\
&= 
\iint \frac{\psi(xh^{-1} {}^ty)}{\chi(A)} J'(\eta \begin{bmatrix}
A h^{-1} & & &  \\
& 1 & x +z(\bar{\tau})^* & \\
&  &A &  \\
& & &1  \\
\end{bmatrix}
\begin{bmatrix}
\delta & & &  \\
&z_{0}^{-1} & & \\
& & d &  \\
& & &1  \\
\end{bmatrix}) \dr A \dr x\\
&= 
\iint \frac{\psi((x -z(\bar{\tau})^*){}^ty')}{\chi(A)} J'(\eta\begin{bmatrix}
A & & &  \\
& 1 & x  & \\
&  &A &  \\
& & &1  \\
\end{bmatrix}
\begin{bmatrix}
\delta & & &  \\
&z_{0}^{-1} & & \\
& & d &  \\
& & &1  \\
\end{bmatrix}) \dr A \dr x,
\end{align*} 
where we write $y' = y{}^th^{-1}$, and the integral $A$ (resp. $x$) is over $P_m \bs G_m$ (resp. $F^m$).
Since $\eta$ lies in $P_{n+2}(\oo)$, we have 
\begin{align*}
& J'(\eta\begin{bmatrix}
A & & &  \\
& 1 & x  & \\
&  &A &  \\
& & &1  \\
\end{bmatrix}
\begin{bmatrix}
\delta & & &  \\
&z_{0}^{-1} & & \\
& & d &  \\
& & &1  \\
\end{bmatrix}) = 
J'(\begin{bmatrix}
A & & &  \\
& A &  & \\
&  x &1 &  \\
& & &1  \\
\end{bmatrix}
\begin{bmatrix}
\delta & & &  \\
&d & & \\
& &z_{0}^{-1}&  \\
& & &1  \\
\end{bmatrix}\eta) \\
&= 
J'(\begin{bmatrix}
A \delta & & &  \\
& A d &   & \\
& &z_{0}^{-1} &  \\
& & &1  \\
\end{bmatrix}
\begin{bmatrix}
1_m & & &  \\
& 1_m & & \\
& x' & 1 &  \\
& & &1  \\
\end{bmatrix}),
\end{align*} 
where $x' = z_0x d$. 
Therefore, $\psi(z(\bar{\tau})^*{}^ty')\If'_{y,z}$ equals 
\begin{align*}
\chi(z_0)\int_{F^m}\int_{P_m \bs G_m} \frac{\psi(x{}^ty')}{\chi(A)} J'(\begin{bmatrix}
A d^{-1} \delta & & &  \\
& A&   & \\
&&z_{0}^{-1} &  \\
& & &1  \\
\end{bmatrix}
\begin{bmatrix}
1_m & & &  \\
& 1_m & & \\
& x' & 1 &  \\
& & &1  \\
\end{bmatrix}) \dr A \dr x.
\end{align*} 
Put
\begin{align*}
\If_{x,z}^\circ = 
\int_{P_m \bs G_m} \chi(A)^{-1} J'(\begin{bmatrix}
A d^{-1} \delta & & &  \\
& A &   & \\
& &z_{0}^{-1} &  \\
& & &1  \\
\end{bmatrix}
\begin{bmatrix}
1_m & & &  \\
& 1_m & & \\
& x' & 1 &  \\
& & &1  \\
\end{bmatrix}) \dr A. 
\end{align*} 
Observe that 
\begin{align}
\begin{split}
\If_z = \int_{F^m} \If'_{y,z} \dr y &= \chi(z_0) \int_{F^m}\int_{F^m} \psi(-z(\bar{\tau})^*{}^ty') \psi(x{}^ty') \If_{x, z}^\circ \dr x \dr y \\
&= \chi(z_0) \int_{F^m}\int_{F^m} \psi(-z(\bar{\tau})^*{}^ty)\psi(x{}^ty) \If_{x, z}^\circ \dr x \dr y.
\end{split} \label{eq:Ifz=}
\end{align} 
Abbreviate
\begin{align*}
\tau' = \tau_{x'}^{o(d^{-1} \delta)},\ \s' = \s_{x'},\  x_0' = {x'}_{\s'(m)}. 
\end{align*} 
By Prop. \ref{prop:LST}, there exist $h' \in (K_m^1)^\delta, c \in D_m$ and $k' \in K_{m+1}^1$ such that 
\begin{align}
\bu(x') = \ac{h'} u(x'(\tau_{x'}^{})^*) \begin{bmatrix}
c&  \\
& x_0'
\end{bmatrix}k'. \label{eq:defc}
\end{align}
Using this expression, we compute
\begin{align}
& \If_{x,z}^\circ = \int \chi(A)^{-1} J'(\begin{bmatrix}
A  h'^{-1} d^{-1}\delta & & &  \\
& A&   & \\
& &z_{0}^{-1} &  \\
& & &1  \\
\end{bmatrix}
\begin{bmatrix}
1_m & & &  \\
& 1_m &{}^t x'(\tau')^* & \\
& & 1 &  \\
& & &1  \\
\end{bmatrix}
\begin{bmatrix}
1& & &  \\
& c & & \\
& &x_0' &  \\
& & & 1 \\
\end{bmatrix}) \dr A \notag \\
&= \int \frac{\psi(A_m {}^t z_0x'(\tau')^*)}{\chi(A)} J'(\begin{bmatrix}
A d^{-1}\delta & & &  \\
& A c &   & \\
& &x_0' z_{0}^{-1} &  \\
& & &1  \\
\end{bmatrix}) \dr A \notag \\
&= 
\chi(x_0')^{-1} \int \frac{\psi((Ac^{-1})_m z_0 {}^tx'(\tau')^*)}{\chi(A)} J'(\begin{bmatrix}
A c^{-1}d^{-1} \delta & & &  \\
& A &   & \\
& &x_0' z_{0}^{-1} &  \\
& & &1  \\
\end{bmatrix}) \dr A \notag \\
&= 
\chi(x_0')^{-1} \int \frac{\psi(A_m {}^t(z_0 x'(\tau')^*c^{-1}))}{\chi(A)}J'(\begin{bmatrix}
A (dc)^{-1}\delta & & &  \\
& A &   & \\
& &x_0' z_{0}^{-1} &  \\
& & &1  \\
\end{bmatrix}) \dr A \label{eq:WMintfn}
\end{align} 
where integrals in $A$ are over $P_m \bs G_m$.
By Prop. \ref{prop:Aform}, at (\ref{eq:WMintfn}), we may assume that
\begin{align}
o(x_0') = o(z_{0}). \label{eq:zx0}
\end{align}

From now, consider the case where $\chi$ is ramified.
Assume that the integrated function at (\ref{eq:WMintfn}) is not zero at $A \in P_m \bs G_m$.
Then, by Prop. \ref{prop:Aform}, we may assume that $A$ is in the form of $A_w(t),w \in \Sf_m, t \in (\oo^\tm)^m$ (c.f. (\ref{def:Awt})).
Put
\begin{align*}
\Jf_{x,z}(t) = \chi(A_w(t))^{-1} J'(\begin{bmatrix}
A_w(t) (dc)^{-1}\delta & & &  \\
& A_w(t) &   & \\
& &x_0' z_{0}^{-1} &  \\
& & &1  \\
\end{bmatrix}).
\end{align*} 
Observe that, for $s_1, \ldots, s_m \in \oo^\tm$, 
\begin{align}
\begin{split}
\Jf_{x,z}((s_1t_1, \ldots, s_m t_m)) &= \Jf_{x,z}(t) \prod_{i = 1}^m \chi(s_i)^{-1} \\
 \Jf_{(s_1x_1, \ldots, s_m x_m),z}(t) &= \Jf_{x, (s_1z_1, \ldots, s_m z_m)}(t) = \Jf_{x,z}(t).
  \end{split}
\label{eq:Jfxz}
\end{align}
Assume that $\If_{x,z}^\circ$ is not zero.
Then, by (\ref{eq:GSsup}), we have an identity
\begin{align*}
\{o((z_0 x'(\tau')^*c^{-1})_1, \ldots,o((z_0 x'(\tau')^*c^{-1})_m \} = \{ -e, \ldots, -me \}
\end{align*} 
as sets.
Thus, $\tau'$ is the full-set.
However, from (\ref{def:dxi}), it follows that 
\begin{align*}
(x'^*c^{-1})_i = 
\begin{cases}
1 & \mbox{if $i = \s'(1)$} \\
x'^{-1}_{\s'(\s'^{-1}(i) -1)} & \mbox{otherwise}. 
\end{cases}
\end{align*}
Taking (\ref{eq:zx0}) into account, we conclude that 
\begin{align*}
o(z_0) = -me, \ o(x_{\s'(i)}') = - ie, \ i \in \{1,\ldots, m \}.
\end{align*} 
Therefore,
\begin{align*}
o(c_{i,i}) = e, \ i \in \{1,\ldots, m \}.
\end{align*}
\begin{prop}
If the integrated function at (\ref{eq:WMintfn}) is not zero, then $\bar{\tau}$ is the full-set.
Thus $z$ lies in $(F \setminus \oo)^m$.
\end{prop}
\begin{proof}
Assume that $\bar{\tau}$ is not the full-set, then some $d_{ii}$ is $1$ and therefore $o(d_{ii})$ is zero.
But, by Prop. \ref{prop:Aform}, we need $o(cd \delta_m^\sharp) = o((\delta_m^\sharp)^{r})$ for some $r \in \Sf_m$, which cannot happen.
\end{proof}
\nid
By (\ref{eq:Jfxz}),  
\begin{align*}
\If_{x,z}^\circ &= \chi(x_0')^{-1} \int_{(\oo^\tm)^m} \psi(A_w(t)_m {}^t(z_0 x'^*c^{-1})) \Jf_{x,z}(t) \dr t \\
&= \gamma_{x,z} \Jf_{x,z}(1),
\end{align*} 
where 
\begin{align*}
\gamma_{x,z} = \frac{\gf(\chi^{-1},\psi_{\Vp^{w_-(\s'(1))e}z_0})}{\chi(x_0')}\prod_{i = \{1, \ldots,m \} \setminus \{\s'(1)\}} \gf\left(\chi^{-1}, \psi_{\Vp^{w_-(i)e}z_0/x'_{\s'(\s'^{-1}(i) -1)}} \right).
\end{align*} 
Since we have assumed $\If_{x,z}^\circ \neq 0$, we need 
\begin{align*}
w_-(i) &= m -\s'^{-1}(i), \  i \in \{1,\ldots, m \}
\end{align*} 
by (\ref{eq:GSsup}).
By (\ref{eq:GSprop}) and (\ref{eq:Jfxz}), for $s_1, \ldots, s_m \in \oo^\tm$, 
\begin{align*}
\gamma_{(s_1x_1, \ldots, s_mx_m),z} = \gamma_{x,z}\prod_{i=1}^m \chi(s_i)^{-1} \\
\gamma_{x,(s_1z_1, \ldots, s_mz_m)} = \gamma_{x,z}\chi(s_{\s(m)})^{-1}.
\end{align*} 
By this property, we have
\begin{align*}
\int_{F^m} \psi(x{}^ty) \If_{x,z}^\circ \dr x &= \gamma_{\xb,y,z}' \Jf_{\xb,z}(1), \\
\gamma_{\xb,y,z}' &= \gamma_{\xb,z} \prod_{i=1}^m \gf(\chi^{-1},\psi_{y_i \xb_i})
\end{align*} 
where $\xb \in \supp(\If_{x,z}^\circ)$. 
By (\ref{eq:GSprop}) and (\ref{eq:Ifz=}),
\begin{align*}
\If_z &= \chi(z_0) \int_{F^m} \psi(-z^*{}^ty)\gamma_{\xb,y,z}' \Jf_{\xb,z}(1) \dr y.\\
&= \chi(z_0) \Jf_{\xb,z}(1) \gamma''_{\xb,z}, \\
\gamma''_{\xb,z} &= \gamma_{\xb, \yb,z}' \prod_{i=1}^m \gf(\chi^{-1},\psi_{-\yb_iz_i^{-1}}),
\end{align*}  
where $\yb \in \supp(\prod_{i=1}^m \gf(\chi^{-1},\psi_{y_i \xb_i}))$. 
By (\ref{eq:GSprop}) again,
\begin{align*}
\Xi_{n+2}J^\circ(g) &= \int_{F^m} \psi(\bu_m{}^tz) \chi(z_0) \Jf_{\xb,z}(1) \gamma''_{\xb,z} \dr z \\
&= \chi(\Vp^{-me}) \gf(\chi^{-1},\psi_{\Vp^{-e}})^m \gamma''_{\xb,\zb} \Jf_{\xb, \zb}(1),
\end{align*} 
where 
\begin{align*}
\zb = (\Vp^{-me}, \ldots, \Vp^{-e}).
\end{align*} 
\begin{prop}
If $\chi$ is ramified, then $\Xi_{n+2}J^\circ(g)$ is nonzero.
\end{prop}
\begin{proof}
By the above argument, we have 
\begin{align*}
\Xi_{n+2}J^\circ(g) = \int_{Z} \int_Y \int_X \psi(\bu_m{}^tz-z^*{}^ty+ x{}^ty)\If_{x,z}^\circ \dr x \dr y \dr z,
\end{align*} 
where
\begin{align*}
Z &= \Vp^{-me} \oo^\tm \tm \Vp^{1-me} \oo^\tm \tm \cdots \tm \Vp^{-e} \oo^\tm\\
Y &= \Vp^{-(m+1)e} \oo^\tm \tm \Vp^{-me} \oo^\tm \tm \cdots \tm \Vp^{-2e} \oo^\tm\\
X &= \Vp^{me} \oo^\tm \tm \Vp^{(m-1)e} \oo^\tm  \cdots \tm \Vp^{e} \oo^\tm.
\end{align*} 
If $(x,y,z)$ lies in $X \tm Y \tm Z$, then $\If_{x,z}^\circ$, $\gamma_{x,z}, \gamma_{x,y,z}', \gamma''_{x,z}$ are nonzero, and so is $\Xi_{n+2}J^\circ(g)$ (in this case, $\s, \s' \in \Sf_m$ are the identity.).
\end{proof}
From now consider the case where $\chi$ is unramified.
Since we assumed that the integrated function (\ref{eq:WMintfn}) is nonzero, $\det(cd)$ lies in $\oo^\tm$ by (\ref{eq:zx0}).
By Theorem \ref{thm:unrpsi}, for $A \in G_m$ at (\ref{eq:WMintfn}), there exists a $p \in P_m$ such that 
\begin{align*}
\begin{bmatrix}
A (dc)^{-1} &  \\
& A
\end{bmatrix} \in \begin{bmatrix}
p & *  \\
& p
\end{bmatrix} P_{n}(\oo).
\end{align*} 
Therefore, it holds that
\begin{align}
o(c_{i,i}) = -o(d_{i,i}), i \in \{1,\ldots, m\}, \label{eq:ocod}
\end{align}
and $A$ lies in $P_mK_m$.  
\begin{prop}\label{prop:zd}
Under the above assumption, we have the followings.
\begin{enumerate}[i)]
\item $o(x_0') (=o(z_0)) = -1$. 
\item $\bar{\tau} = \{m \}$.
\item Let $i$ be the first number such that $z_i \not\in \oo$.
Then,
\begin{align*}
d = \diag(\overbrace{1, \ldots, 1}^{i-1}, z_i^{-1}, \overbrace{1, \ldots, 1}^{m-i}).
\end{align*} 
\end{enumerate}
\end{prop}
\begin{proof}
i) Let $j$ be the first number (may be $1$) such that $j -1 \not\in \tau'$ and $j \in \tau'$ (such a $j$ exists, since $\tau'$ contains $m$.).
Then, 
\begin{align*}
(z_0 x'(\tau')^*c^{-1})_{\s'(j)} = z_0.
\end{align*}
Let $w_j$ be the longest Weyl element of $\Sf_j$.
The coset space $P_m \bs P_m K_m \simeq P_m(\oo) \bs K_m$ is realized as the set of matrices: 
\begin{align*}
\begin{bmatrix}
1_{m-1}&  \\
c & d
\end{bmatrix} 
\begin{bmatrix}
1_{m-j}&  \\
& w_j
\end{bmatrix}, \  j \in \{0, \ldots,m \}
\end{align*}
with $c_1, \ldots, c_{m-j-1} \in \oo, c_{m-j}, \ldots, c_{m-1} \in \p, d \in \oo^\tm$.
Therefore, if $o(z_0)$ is less than $-1$, then 
\begin{align*}
\int_{P_m \bs P_m K_m} \psi(A_m {}^t(z_0 x'(\tau')^*c^{-1})) \dr A = 0, 
\end{align*} 
and (\ref{eq:WMintfn}) is also zero.
Hence i).
Therefore, $o(z_{\s(j)}) \ge -1$ for all $j$.
Take the last $j \neq m$ such that $j \in \tau$ if exists.
Then, $o(z_{\s(j)})$ equals $-1$ by the definition of $\bar{\tau}$.
Thus $o(z_{0}) - o(z_{\s(j)})$ equals $0$, conflicting with the definition of $\bar{\tau}$.
Hence ii).
Now iii) follows immediately from (\ref{def:d}).
\end{proof}
\nid
Set 
\begin{align*}
f^i = (\overbrace{0, \ldots, 0}^{i-1}, 1, \overbrace{0, \ldots, 0}^{m-i}) \in \Z^m.
\end{align*} 
\begin{prop}\label{prop:x}
Let $i$ be as in Prop. \ref{prop:zd}.
Then, 
\begin{align*}
\tau_{x'}^{f^i} = \{ m \}, \ 
x'_j \in 
\begin{cases}
\Vp^{-1}\oo^\tm  & \mbox{if $j = i$} \\
\oo & \mbox{otherwise.} 
\end{cases}
\end{align*} 
\end{prop}
\begin{proof}
From (\ref{eq:ocod}), and the definition (\ref{eq:defc}) of $c$, it follows that $o(x_i') = -1$, and $i \in \tau' = \tau_{x'}^{f^i}$.
Set 
\begin{align*}
 r &= |\{j \in \{1,\ldots, i \} \mid o(x_j) =-1 \}| \\
 s &= |\{j \in \{i+1,\ldots, m \} \mid o(x_j) =-1 \}| \\
 t &= m -r-s.
 \end{align*}
Observe that $t$ is the number of $j$ such that $x_j \in \oo$.
By the definition of $\s' = \s'_{x'}$, we have $\s'(t+r) = i$.
Assume that $s > 0$.
Then, 
\begin{align*}
o(x_{\s'(t+r)}) -o(x_{\s'(t+r+1)}) =0 >-1 =  f_{\s'(t+r+1)} -f_{\s'(t+r)}, 
\end{align*}
conflicting with the condition $i \in \tau'$.
Hence $s = 0$.
Assume that $r \ge 2$.
Then, since
\begin{align*}
o(x_{\s'(t+r-1)}) -o(x_{\s'(t+r)}) =0 < 1 =  f_{\s'(t+r)} -f_{\s'(t+r-1)}, 
\end{align*}
we have $t+r-1 \in \tau'$ and $o(c_{i,i}) = o(x'_{t+r-1}) - o(x'_i) = 0$, conflicting with (\ref{eq:ocod}).
Hence $r = 1$.
Now the assertion is obvious.
\end{proof}
\nid
Define subsets $S_{i,l}, R_{i,l} \subset F^m$ as follows.
\begin{align*}
S_{i,l} &= \overbrace{\oo \tm \cdots \tm \oo}^{i-1} \tm \p^{l} \tm \overbrace{\oo \tm \cdots \tm \oo}^{m-i}, \\
R_{i,l} &= S_{i,l} \setminus S_{i,l-1} = \overbrace{\oo \tm \cdots \tm \oo}^{i-1} \tm \Vp^{l} \oo^\tm \tm \overbrace{\oo \tm \cdots \tm \oo}^{m-i}.
\end{align*} 
Let $i$ be as in Prop. \ref{prop:zd}.
Then, it holds that 
\begin{align*}
z_0x'(\tau')^*c^{-1} = (\overbrace{1, \ldots, 1}^{i-1}, z_0^{-1}, \overbrace{1, \ldots, 1}^{m-i}) \in R_{i,-1}, 
\end{align*}
and that, independent from $i$, 
\begin{align*} 
\int_{P_m\bs G_m} \psi(A_m {}^t(z_0 x'(\tau')^*c^{-1})) \dr A = -q^{-m}.
\end{align*}
From (\ref{eq:WMintfn}), we obtain 
\begin{align*}
\If_{x,z}^\circ = 
\begin{cases}
-\chi(\Vp) q^{-m} & \mbox{if $o(z_0) = -1$, and $x' \in R_{i,-1}$} \\
0 & \mbox{otherwise} 
\end{cases}
\end{align*} 
for $z$ outside of $\oo^m$.
Reminding $x = z_0^{-1}x'd$, we obtain from (\ref{eq:Ifz=}) 
\begin{align}
\If_z = 
\begin{cases}
- q^{-m} \sum_{i=1}^m \int_{F^m}\int_{R_{i,1}}\psi(- z(\{m\})^* {}^t  y) \psi(x{}^ty) \dr x \dr y & \mbox{if $o(z_0) = -1$} \\
0 & \mbox{otherwise} 
\end{cases}
\label{eq:Izfinunr}
\end{align}
for $z$ outside of $\oo^m$.
\begin{prop}\label{prop:}
Assume that the number of $j \in \{1, \ldots, m\}$ such that $o(z_j) = -1$ is greater than $1$. 
Then $\If_z$ is zero. 
\end{prop}
\begin{proof}
Let $i$ be as in Prop \ref{prop:zd}.
Viewing (\ref{eq:Izfinunr}), we consider the condition for 
\begin{align*}
\int_{R_{i,1}} \psi(x{}^ty) \dr x \neq 0.
\end{align*} 
It is that $y$ lies in $S_{i,-2}$. 
But since $z(\{m\})^*$ equals
\begin{align*}
 (\overbrace{0, \ldots, 0}^{j-1}, z_j^{-1}, \overbrace{0, \ldots, 0}^{m-j})
\end{align*} 
for the last number $j \neq i$, under our assumption, such that $o(z_j) = -1$, we have $\psi(z(\{m\})^*{}^ty) = 1$ for $y \in S_{i,-2}$. 
Now consider 
\begin{align*}
& \int_{S_{i,-2}} \int_{R_{i,1}} \psi(z(\{m\})^*{}^ty) \psi(x{}^ty) \dr x \dr y \\
&= \int_{S_{i,-2}} \int_{R_{i,1}} \psi(x{}^ty) \dr x \dr y = \int_{R_{i,1}} \int_{S_{i,-2}}  \psi(x{}^ty)  \dr y \dr x = 0.
\end{align*}
Hence the assertion.
\end{proof}
\nid
We compute that 
\begin{align*}
\int_{S_{i,-2}} \int_{R_{i,1}} \psi(- z(\{m\})^* {}^t  y) \psi(x{}^ty) \dr x \dr y = \frac{1-q}{q^2}
\end{align*} 
independent from $i$.
From (\ref{eq:WMintfn}),  
\begin{align*}
\If_z = \frac{m(q-1)}{q^{m+2}}
\end{align*} 
for $z \in R_{i,-1}$.
Combining (\ref{eq:Ifyz}), we obtain finally 
\begin{align*}
\Xi J^\circ(1_{n+2}) &= \int_{\oo^m} \psi(z_m) \If_z \dr z + \int_{F^m \setminus \oo^m} \psi(z_m) \If_z \dr z \\
&= 1 + ((m-1)q - m)\frac{m(q-1)}{q^{m+2}} \\
& \neq 0.
\end{align*} 
Thus, $\Xi_{n}$ is nontrivial.
Since $\Xi_r$ for $r <n$ is an integral over a subset of $P_r$, it is also applied to functions on $P_n$.
We define the $P_n$-linear map $\Omega_n$ by the iteration of the integrals:
\begin{align*}
\Omega_{n} = \Xi_4 \circ \cdots \circ \Xi_{n} \in \Hom_{P_n}(\Ib_{\psi_n}(\chi),\psi_n).
\end{align*} 
\begin{prop}\label{prop:Ln}
With the above notation, we have the followings. 
\begin{enumerate}[i)]
\item $\Omega_n$ is a nonzero constant multiple of $(\Lambda_n^\chi)^{-1}$.
\item $\Omega_n(J_{n,\chi}^\circ)$ is  a nonzero constant multiple of $\xi_n$.
\item $\Hom_{P_n}(\Ib_{\psi_n}(\chi),\psi_n) = \C \Omega_n$.
\end{enumerate}
\end{prop}
\begin{proof}
By (\ref{eq:thnchi}) and Cor. \ref{Cor:PreSLam}, $\Hom_{P_{n}}(\Ib_{\psi_{n-2}}(\chi)_n, \Ib_{\psi_{n-2}}(\chi')_n)$ for arbitrary $\chi$ and $\chi'$ is spanned by 
\begin{align*}
\Theta_{\chi,\chi'}^n:= (\theta_n^{\chi'})^{-1} \circ \Lambda_n^{\nu^{-1}\chi'} \circ (\Lambda_n^{\nu^{-1}\chi})^{-1} \circ\theta_n^\chi,
 \end{align*}
which is bijective.
We have
\begin{align*}
\Theta_{\chi,\chi'}^n \circ \phi_{n-2}^{0} (J_{n-2,\chi}^\circ) = c \phi_{n-2}^{0}(J_{n-2,\chi'}^\circ), \ c \neq 0.
\end{align*}
Since $\Xi_n \circ \theta_n^{\chi} \in \Hom_{P_{n}}(\Ib_{\psi_{n-2}}(\chi)_n, \Ib_{\psi_{n-2}}(\nu^2\chi)_n)$ is nontrivial, it is a nonzero constant multiple of $\Theta_{\chi,\chi'}^n$, and we have
\begin{align*}
\Xi_n \circ \theta_n^{\chi} \circ \phi_{n-2}^{0} (J_{n-2,\chi}^\circ) = c' \phi_{n-2}^{0}(J_{n-2,\nu^2\chi}^\circ), \ c' \neq 0.
\end{align*} 
Therefore, 
\begin{align*}
\Xi_{n-2} \circ \Xi_n(J_{n,\chi}^\circ) = c' \Xi_{n-2} \circ  \phi_{n-2}^{0}(J_{n-2,\nu^2\chi}^\circ).
\end{align*} 
By iv) of Prop. \ref{prop:welldefJW}, 
\begin{align*}
\Xi_{n-2} \circ \Xi_n(J_{n,\chi}^\circ)(\ac{g}) = c' \phi_{n-2}^{(0)} \circ \Xi_{n-2} (J_{n-2,\nu^2\chi}^\circ)(\ac{g}), \ g \in P_{n-2}.
\end{align*} 
By the above computation, $\Xi_{n-2} (J_{n-2,\nu^2\chi}^\circ)$ takes a nonzero value at $\ac{g}_{n-2}$, and so does $\Xi_{n-2} \circ \Xi_{n}(J_{n,\chi}^\circ)$.
Thus $\Xi_{n-2} \circ \Xi_{n}$ is nontrivial.
Iterating such arguments, we find that $\Omega_n$ is nontrivial.
Now the assertions follow from the uniqueness of pre-Shalika models, and the irreducibilities of $\psi_n, \Ib_{\psi_n}(\chi)$.
\end{proof}
\nid 
Since the $P_n$-linear map $\Omega_n$ is an integral over a subset of $P_n$, it can be applied to Shalika forms on $G_n$, and obviously, 
\begin{align*}
\Omega_n \in \Hom_{N_n}(\Sb_\pi(\chi), \psi) \simeq \Hom_{G_n}(\Sb_\pi(\chi), \In_{N_n}^{G_n}(\psi)).
\end{align*} 
We cannot say $\Omega_n$ is nontrivial without assumption at this time.
However, we have:
\begin{thm}\label{thm:Ln}
Let $\pi \in \Ir(G_n)$ be generic.
Assume (\ref{ass:J1nz}).
Then, $\Omega_n$ is a nontrivial Whittaker model of $\pi$.
Additionally, assume $L(s,\pi) = 1$.
If $J \in \Sb_\pi(\chi)$ is a Shalika newform, then $J|_{P_n}$ is a nonzero constant multiple of $J_{n,\chi}^\circ$.
\end{thm}
\begin{proof}
The first assertion follows from Prop. \ref{prop:SPS}, immediately.
Since $\Omega_n$ is an integral over a subset of $P_n$, 
\begin{align*}
\Omega_n(J|_{P_n}) = \Omega_n(J)|_{P_n}.
\end{align*}
Let $W^{new}$ be the Whittaker newform of $\pi$ (c.f. sec. \ref{sec:pre}).
Since $\Omega_n$ as a Whittaker model is bijective, $\Omega_n(J) = c W^{new}, c \neq 0$.
By Prop. \ref{prop:LW}, $\Omega_n(J)|_{P_n} = c \xi_n$.
By ii) of Prop. \ref{prop:Ln},
\begin{align*}
J|_{P_n} = \Omega_n^{-1}(\Omega_n(J)|_{P_n}) = \Omega_n^{-1}(c \xi_n) = c' J_{n,\chi}^\circ, c' \neq 0.
\end{align*} 
This proves the second assertion.
\end{proof}
\begin{prop}\label{prop:LW}
Let $\pi \in \Ir(G_n)$ be generic.
If $L(s,\pi) = 1$, then 
\begin{align*}
W^{new}|_{P_n} = \xi_n.
\end{align*}
\end{prop}
\nid
This will be proved in the next section.
\section{Hecke operators}\label{sec:HO}
To prove Prop. \ref{prop:LW}, we introduce some Hecke operators.
Let
\begin{align*}
\Z^r_+ = \{f \in \Z^r \mid f_1 \ge \cdots \ge f_r \ge 0 \} 
\end{align*} 
equipped with the lexicographic order as follows.
If $f,f' \in \Z_+^r$ are different, then, for a unique $i \in \{1, \ldots, r\}$, it holds that $f_{i+1} = f_{i+1}', \ldots, f_r = f_r'$, and $f_i \neq f_i'$.
In this case, if $f_i < f_i'$, then we write  $f \prec f'$.
Let $\Hc_r$ denote the $K_r$ invariant Hecke algebra of $G_r$.
Let $T_f \in \Hc_r$ denote the Hecke operator which acts on right $K_r$-invariant functions $\xi$ on $G_r$ by 
\begin{align*}
T_f\xi(g) = \int_{K_r} \Ch(g; K_r \Vp^f K_r)\xi(g)\dr g, 
\end{align*} 
where $\dr g$ indicates the Haar measure on $G_r$ such that $\vol(K_r) =1$. 
An elementary computation shows that
\begin{align}
T_f\xi(g) = \sum_{\s \in \Sf_r} \sum_{u \in \Nc(f^\s)} \xi(g u \Vp^{f^{\s}}), \label{eq:Tfxi}
\end{align} 
where $f^\s$ indicates $(f_{\s(1)}, \ldots, f_{\s(r)})$, and $\Nc(f^\s)$ is a certain finite subset of $N_r \cap K_r$.
For $i \in \{1, \ldots, r \}$, let 
\begin{align*}
h_i = (\overbrace{1, \ldots, 1}^{i}, 0, \ldots, 0) \in \Z_+^r. 
\end{align*} 
For a positive integer $c$, let $\Gamma_{r+1}(c) = \Gamma(c) \subset G_{r+1}$ be the open compact subgroup defined at (\ref{def:Gc}), and $\Hc_{r+1}^c$ denote the $\Gamma_c$ invariant Hecke algebra of $G_{r+1}$.
Let $T_f^c \in \Hc_{r+1}^c$ denote the Hecke operator the Hecke operator which acts on right semi-$\Gamma(c)$-invariant functions $\xi'$ on $G_r$ by 
\begin{align*}
T_f^c \xi(g) = \int_{K_r} \Ch(g; \Gamma(c) \ac{\Vp}^f \Gamma(c))\xi(g)\dr g, 
\end{align*} 
where $\dr g$ indicates the Haar measure on $G_r$ such that $\vol(\Gamma(c)) =1$. 
An elementary computation shows that for a positive integer $a$, $T_{ah_i}^c$ acts on right semi-$\Gamma(c)$-invariant functions $\xi'$ by
\begin{align}
T_{ah_i}^c \xi'(g) = \sum_{\s \in \Sf_r} \sum_{u \in \Nc'(ah_i^\s)}\xi'(gu \begin{bmatrix}
\Vp^{ah_i^\s} &  \\
& 1
\end{bmatrix}), \label{eq:Thi}
\end{align} 
where $\Nc'(ah_i^\s)$ is a certain finite subset of $N_{r+1} \cap K_{r+1}$.
Let $\pi \in \Ir(G_{r+1})$ be generic.
Assume (\ref{ass:psi}) for $\psi$.
If $W$ is a right semi-$\Gamma(c)$-invariant Whittaker form of $\pi$, then 
\begin{align}
\ac{W}(\Vp^f) = 0,  \ \ \mbox{if $f \not\in \Z^r_+$} \label{eq:WVfnz}
\end{align} 
by Lemma \ref{lem:vanishlem}.
Additonally, if $f$ lies in $\Z_+^r$, then it holds that 
\begin{align}
T_{a h_i}^{c} W(\Vp^f) = \sum_{\s \in \Sf_r} c_\s \ac{W}(\Vp^{f+ ah_i^\s}),  c_\s> 0.\label{eq:ThiW}
\end{align} 
by (\ref{eq:Thi}) and (\ref{ass:psi}).

We prove Prop. \ref{prop:LW} by induction.
Let $W = W^{new}$ be the canonical Whittaker newform of $\pi$.
If $f = (1,0 ,\ldots, 0)$, then $\ac{W}(\Vp^f)$ is a constant multiple of $T_{h_1}W(1_{r+1})$ by (\ref{eq:ThiW}).
The statement in this case follows from Theorem 1.1. of \cite{K-Y}.
Assuming the statement for all nonzero $f' \prec f$ is true, we will show that for $f$ is true.
Let $s$ be the maximal number such that $f_s \neq 0$.
Write 
\begin{align*}
f_- = f -f_s h_s (\in \Z_+^r).
\end{align*} 

Suppose that $f_- = 0$.
Then, $\ac{W}(\Vp^f)$ is a nonzero constant multiple of $T_{f_sh_s}^c W(1_{r+1})$ by (\ref{eq:WVfnz}) and (\ref{eq:ThiW}).
The statement follows from the same theorem.

Suppose that $f_- \neq 0$.
By the induction hypothesis, $\ac{W}(\Vp^{f_-}) = 0$.
By (\ref{eq:ThiW}),  
\begin{align*}
0 = T_{f_sh_s}^cW(\ac{\Vp}^{f_-}) = \sum_{\s \in \Sf_r} c_\s W(\ac{\Vp}^{f_-+ f_s h_s^\s}).
\end{align*} 
If $f_-+ f_s h_s^\s$ lies in $\Z_+^r$ and $h_s^\s \neq h_s$, then $f_-+ f_s h_s^\s \prec f$.
Now the statement follows from the same theorem and the induction hypothesis.
This establishes the induction step.
\section{Zeta integrals}\label{sec:ZI}
Let $\psi$ be satisfying $(\ref{ass:psi})$.
Let $n = 2m$ be an even integer, and $\pi \in \Ir(G_n)$ be generic with a Shalika model relevant to $\chi$.
Let $J^{new}$ be a Shalika newform (recall this is unique up to constant multiples).
Assume that $J^{new}$ does not vanish at $g_n$ of Theorem \ref{thm:main}. 
This condition is empty if $\pi$ is supercuspdal.
Assume that 
\begin{align*}
\cf_\pi \ge me. 
\end{align*} 
Put 
\begin{align*}
l = \cf_\pi - (m-1)e \ (\ge e).
\end{align*} 
Let $\K(\cf_\pi)$ be as in sec. \ref{sec:main}.
First of all, we construct a $J_\pi \in \Sb_\pi(\chi)$ such that 
\begin{align*}
J_\pi(1_n) = 1, \pi(k)J_\pi = \chi \circ \det(d_k) J_\pi, k \in \K(\cf_\pi)
\end{align*} 
where $d_k$ indicates the $m \tm m$ block matrix of $k$ in the lower right corner.
In the case where $\chi$ is unramified, there is nothing to do.
Suppose that $\chi$ is ramified.
We use the following lemma.
\begin{lem}\label{lem:ftrint}
Let $\Omega$ be a field.
Let $f$ be a $\Omega$-valued function on a group $G$ such that $f(gk) = \xi(k) f(g)$ for a subgroup $K$ and a homomorphism $\xi: K \to \Omega^\tm$.
Then, we have the followings:
\begin{enumerate}[i)]
\item For the right translation $f^h$ by $h \in G$, 
\begin{align*}
f^h(g k) = \xi(h^{-1}kh) f^h(g), \ k \in K^h.
\end{align*} 
\item Let $K'$ be a subgroup containing $K$.
Assume that $\xi$ is extended to $\xi':K' \to \Omega^\tm$.
Then, 
\begin{align*}
f'(g) := \int_{K'/K} \xi'(k')^{-1} f(gk') \dr k' 
\end{align*} 
satisfies 
\begin{align*}
f'(gk) = \xi'(k') f(g), \ k' \in K'.
\end{align*} 
\end{enumerate}
\end{lem}
\begin{proof}
Obvious.
\end{proof}
\nid
Let $r$ be a positive integer $r$.
For a set $A$, let $(A)^r$ denote the $m$-tuple product of $A$. 
Let $M_r$ and $\Oc_r$ denote the ring of $r \tm r$ matrices with entries in $F$ and $\oo$, respectively.
If a subgroup $K \subset G_n$ consists of matrices 
\begin{align*}
\begin{bmatrix}
a& b  \\
c& d
\end{bmatrix}, a \in \Af, b \in \Bf, c \in \Cf, d \in \Df
\end{align*} 
for subsets $\Af,\Bf,\Cf,\Df \subset M_m$, we call $K$ the subgroup relevant to $\Af,\Bf,\Cf,\Df$.
Let $\delta_m$ be the diagonal matrix defined in $(\ref{def:dm})$.
For $t \in (\oo^\tm)^{m-1}$, put
\begin{align*}
v_m(t) &= u((\Vp^{-(m-1)e}t_1, \Vp^{-(m-2)e }t_2, \ldots, \Vp^{-e}t_{m-1})) \\
\Jf_t &= \pi(\begin{bmatrix}
1_m &  \\
& v_m(t)
\end{bmatrix})J^{new}.
\end{align*} 
From (\ref{eq:propSh}) and the assumption $J^{new}(g_n) \neq 0$, it follows that $\Jf_1(\ac{\delta}_m) \neq 0$.
By (\ref{eq:propSh}), 
\begin{align*}
\Jf_t(\ac{\delta}_m) = \chi(\prod_{i=1}^{m-1}t_i)\Jf_1(\ac{\delta}_m).
\end{align*} 
Firstly, we set 
\begin{align*}
J_1 = \int_{(\oo^\tm)^{m-1}}\chi(t)^{-1} \Jf_t \dr t,
\end{align*} 
which is not zero at $\ac{\delta}_m$.
For $u = \diag(u_1,\ldots, u_m)$ with $u_i \in \oo^\tm$, it holds that
\begin{align*}
\pi(\begin{bmatrix}
1_m&  \\
& u
\end{bmatrix}) J_1 &= \int_{(\oo^\tm)^{m-1}} \chi(t)^{-1} \pi(\begin{bmatrix}
1_m &  \\
& u
\end{bmatrix}
\begin{bmatrix}
1_m &  \\
& v_m(t)
\end{bmatrix})J \dr t \\
&= \int_{(\oo^\tm)^{m-1}} \chi(t)^{-1} \pi(\begin{bmatrix}
1_m &  \\
& v_m(t')
\end{bmatrix}\begin{bmatrix}
1_m &  \\
& u
\end{bmatrix})J \dr t \\
&= \chi(\prod_{i=1}^{m}u_i) J_1,
\end{align*} 
where $t' = (t_1 u_1/u_m, \ldots, t_{m-1} u_{m-1}/u_m)$. 
By i) of the lemma, $\Jf_t$ is invariant under the subgroup relevant to 
\begin{align*}
& \Oc_m, \ \{b \mid b_{ij} \in \p^{m-1} \}, \ \{c \mid c_{ij} \in \oo; i <m, c_{mj} \in \p^c \}, \\
& \{d \mid d_{ii} \in 1+ \p^e, d_{ij} \in \oo; i < j, d_{mj} \in \p^c; j <m, d_{ij} \in \p^{me}; j < i < m \}.
\end{align*} 
Therefore, it holds that 
\begin{align*}
\pi(k)J_1 = \chi(\prod_{i=1}^m k_{m+i,m+i}) J_1
\end{align*} 
for $k$ lying in the subgroup relevant to 
\begin{align*}
& \Oc_m, \ \{b \mid b_{ij} \in \p^{m-1} \}, \ \{c \mid c_{mj} \in \p^c;  c_{ij} \in \oo,  i <m \}, \\
& 
\{d \mid d_{ij} \in \oo; i \le j, d_{mj} \in \p^c; j <m, d_{ij} \in \p^{me}; j < i < m \}.
\end{align*} 
By (\ref{eq:propSh}),
\begin{align*}
J_1(\ac{\delta}_m) = J_1(\begin{bmatrix}
1_m &  \\
& \delta_m^{-1}
\end{bmatrix}) \neq 0
\end{align*} 
Secondly, we set 
\begin{align*}
J_2 = \pi(\begin{bmatrix}
1_m &  \\
& \delta_m^{-1}
\end{bmatrix})J_1.
\end{align*}
Then $J_2(1_n) \neq 0$, and, by i) of the lemma, 
 \begin{align*}
\pi(k)J_2 = \chi(\prod_{i=1}^m k_{m+i,m+i}) J_2
\end{align*}  
for $k$ lying in the subgroup relevant to 
\begin{align*}
& \Oc_m, \ \{b \mid b_{ij} \in \p^{n-2} \}, \ \{c \mid c_{mj} \in \p^{l}, c_{ij} \in \oo; i <m \}, \\
& \{\begin{bmatrix}
 u& {}^t  (\p^{(m-2)e})^{m-1}  \\
(\p^{c_\pi})^{m-1} & *
\end{bmatrix} \in K_m \mid u_{ij} \in 1 + \p^{(n-4)e}; i \neq j \}.
\end{align*} 
Thirdly, we set 
\begin{align*}
J_3 = \iiint \pi(\begin{bmatrix}
1_m & & \\
 & u & {}^t x \\
 & y &1  \\
\end{bmatrix})J_2 \dr u \dr x \dr y
\end{align*} 
where the integral in $x$ is over $(\oo/\p^{(m-2)e})^{m-1}$, that in $y$ is over $(\p^l/\p^{c_\pi})^{m-1}$, and that in $u$ over 
\begin{align*}
& K_{m-1}/\{u \in K_{m-1} \mid u_{ij} \in 1 + \p^{(n-4)e}; i \neq j \} \\
& \simeq 
\SL_{m-1}(\oo)/\{u \in \SL_{m-1}(\oo) \mid u_{ij} \in 1 + \p^{(n-4)e}; i \neq j \}.
\end{align*} 
By ii) of the lemma, 
\begin{align*}
\pi(k)J_3 = \chi \circ \det(d_k) J_3
\end{align*}  
for $k$ lying in the subgroup relevant to 
\begin{align*}
& \Oc_m, \ \{b \mid b_{ij} \in \p^{n-2} \}, \ \{c \mid c_{mj} \in \p^l;  c_{ij} \in \oo,  i <m \}, \\
& \{d = \begin{bmatrix}
u & {}^t \oo^{m-1} \\
(\p^l)^{m-1} & *
\end{bmatrix} \in K_m \mid u \in K_{m-1}\}.
\end{align*} 
By (\ref{eq:propSh}), 
\begin{align*}
J_3(1_n) 
&= \iiint J_2(\begin{bmatrix}
1_m & & \\
 & u & {}^t x \\
 & y &1  \\
\end{bmatrix}) \dr u \dr x \dr y \\ 
&= \iiint J_2(\begin{bmatrix}
u &{}^tx & \\
y & 1 &  \\
 &  &1_m  \\
\end{bmatrix}) \dr u \dr x \dr y 
\end{align*} 
is nonzero, where the integral region is the same as above.
Finally, we set 
\begin{align*}
J_\pi = c \int \pi(\begin{bmatrix}
1&x  \\
&1
\end{bmatrix})J_3 dx, 
\end{align*} 
where the integral in $x$ is over $\Oc_m/ \{b \mid b_{ij} \in \p^{n-2} \}$, and $c$ is the nonzero constant taken so that $J_\pi(1_n) = 1$.
Now the semi-$\K(\cf_\pi)$-invariance property of $J_\pi$ follows from ii) of the lemma.
Let $\K(\cf_\pi)^*$ and $J_\pi^* \in \Sb_{\pi^\vee}(\chi^{-1})$ be as in sec. \ref{sec:main}.
By i) of the lemma, 
\begin{align*}
\pi^\vee(k)J_\pi^* = \chi \circ \det(d_k) J_\pi^*, k \in \K(\cf_\pi)^*.
\end{align*} 

For a finite dimensional vector space $V$ over $F$, let $\Ss(V)$ denote the Schwartz space of $V$.
Define $\vp_{\cf_\pi} \in \Ss(M_m)$ as follows.
In the case where $\chi$ is unramified, it is the characteristic function of the ring $R_{\cf_\pi}$ (c.f. \ref{eq:Rc}).
Suppose that $\chi$ is ramified.
Define $\chi_0 \in \Ss(F)$ by 
\begin{align*}
\chi_0(x) = \Ch(x;\oo^\tm) \chi(x)
\end{align*} 
where $\Ch$ indicates the characteristic function.
Define $\phi_\chi^\circ \in \Ss(M_{m-1})$ by
\begin{align*}
\phi_\chi^\circ(x) = \prod_{1 \le i \neq j \le m-1}\Ch(x_{ij};\oo) \prod_{i=1}^{m-1} \chi_0(\Vp^e x_{ii}).
\end{align*} 
Define $\phi_{\chi,l} \in \Ss(M_{m})$ by
\begin{align*}
\phi_{\chi,l}(\begin{bmatrix}
x & {}^t y \\
z & w
\end{bmatrix}) = \frac{\Ch(y, z; (\p^{-l})^{m-1} \tm \oo^{m-1}) \chi_0(\Vp^e w)}{\vol(\SL_{m-1}(\oo))}\int_{\SL_{m-1}(\oo)}\phi_\chi^\circ(x  u) \dr u. 
\end{align*} 
Observe that 
\begin{align}
\phi_{\chi,l}(vk) = \chi \circ \det(k) \phi_{\chi,l}(v), \  k \in \Gamma_m(l). \label{eq:phchl}
\end{align} 
Define $\vp_{\cf_\pi} \in \Ss(M_n)$ by 
\begin{align*}
\vp_{\cf_\pi}(\begin{bmatrix}
a &b  \\
c& d
\end{bmatrix}) = \Ch(\begin{bmatrix}
a &  \\
c& d
\end{bmatrix}; R_{\cf_\pi}) \phi_{\chi,l}(b).
\end{align*} 
We have defined $\vp_{\cf_\pi}$ so that $J_\pi \vp_{\cf_\pi}$ is right $\K(\cf_\pi)$-invariant.
Let $Z(s,J_\pi, \vp_{\cf_\pi})$ denote the Godement-Jacquet zeta integral of $J_\pi$ and $\vp_{\cf_\pi}$ defined by
\begin{align*}
Z(s,J_\pi, \vp_{\cf_\pi}) = \int_{G_n} J_\pi \vp_{\cf_\pi}(g) |\det(g)|^s \dr g.
\end{align*} 
For $k \in \Z$, let 
\begin{align*}
B_{m,k}= \{b \in B_m \cap \Oc_m \mid o(\det(b)) = k \},
\end{align*} 
and 
\begin{align*}
c_k = \sum_{B_{m,k}/B_{m,0}} \ac{J}_\pi(b).
\end{align*} 
\begin{prop}\label{prop:ZJPhi}
With the above notation, 
\begin{align*}
Z(s + \frac{n-1}{2}, J_\pi, \vp_{\cf_\pi}) = q^{l(m-1)} \gf(\chi,\psi_{\Vp^{-e}})^{m}\vol(\K(\cf_\pi)) \sum_{i = 0}^\infty c_i q^{i(-s+\frac{1}{2})}.
\end{align*} 
\end{prop}
\begin{proof}
We observe what $g \in G_n$ contributes to the zeta integral.
Using a complete system of representatives for $\Gamma_{n}(\cf_\pi)/\K(\cf_\pi)$, we find by Lemma 2.1. of \cite{K-Y} that we may assume that $g$ is in the form of 
\begin{align}
\begin{bmatrix}
a + bdc & bd  \\
dc & d 
\end{bmatrix} = \begin{bmatrix}
1_m & b   \\
& 1_m 
\end{bmatrix}\begin{bmatrix}
a &  \\
& d
\end{bmatrix}\begin{bmatrix}
1_m &  \\
c & 1_m
\end{bmatrix}, 
\label{eq:ng0K}
\end{align} 
where $a,b,c,d$ are $m \tm m$ block matrices.
We claim that $d$ lies in $\Gamma_m(l)$, using the identity
\begin{align}
\begin{bmatrix}
1_m & b' \\
& 1_m
\end{bmatrix} g = \begin{bmatrix}
(a + b dc) + b' dc & bd + b' d  \\
dc & d 
\end{bmatrix}. \label{eq:dingammal}
\end{align}
If $g$ contributes to the zeta integral, then (\ref{eq:ng0K}) lies in $\supp(\vp_{\cf_\pi})$,  therefore 
\begin{align}
(dc)_{mj} \in \p^l, (dc)_{ij} \in \p^e, i \in \{1, \ldots, m-1\}, j \in \{1, \ldots, m\}, \label{eq:dcmj}
\end{align} 
and we may write
\begin{align*}
d = \begin{bmatrix}
d^\circ & *  \\
* & d_{mm}
\end{bmatrix} \in \begin{bmatrix}
\Oc_{m-1} & {}^t(\oo)^{m-1}  \\
(\p^l)^{m-1} & \oo
\end{bmatrix}.
\end{align*} 
By the $\K(\cf_\pi)$-invariance property of $J_\pi \vp_{\cf_\pi}$, we may assume that $d^\circ$ is an upper trianglar matrix.
It suffices to show that $\det(d)\in \oo^\tm$.
Assume that $\det(d) \in \p$.
Then $d_{kk} \in \p$ for some $k \in \{1, \ldots, m\}$.
Using the assumption (\ref{eq:dcmj}), and that $d^\circ$ is an upper trianglar matrix, we find that it is possible to take a $b'$ so that 
\begin{align*}
b_{kk}' &\in \p^{-1}, b'_{jj} =0, j \in \{1, \ldots, m\} \setminus \{k\},  \\
b'd c &\in \Oc_m, \\
b'd &\in \oo E_{kk}
\end{align*}
where $E_{kk}$ indicates the $k$-th row and $k$-th column matrix unit.
But, since (\ref{eq:dingammal}) also lies in $\supp(\vp_{\cf_\pi})$, and  
\begin{align*}
J_\pi \vp_{\cf_\pi}(\begin{bmatrix}
1_m & b' \\
& 1_m
\end{bmatrix} g) = \psi(b_{kk}') J_\pi \vp_{\cf_\pi}(g)
\end{align*} 
by (\ref{eq:propSh}), $g$ does not contribute.
Hence, the claim.
Now, it is easy to see that 
\begin{align*}
\begin{bmatrix}
&  \\
c & 
\end{bmatrix} = \begin{bmatrix}
&  \\
 & d^{-1} 
\end{bmatrix}\begin{bmatrix}
&  \\
d c & 
\end{bmatrix} \in R_{\cf_\pi}
\end{align*} 
and we may assume that 
\begin{align*}
b_{ii} &\in \Vp^{-e}\oo^\tm, \ i \in \{1, \ldots, m\}, \\ 
b_{mj} &\in \oo, b_{jm} \in \p^{-l}, \ j \in \{1, \ldots, m-1\}
\end{align*} 
by (\ref{eq:phchl}).
Therefore, $bdc$ lies in $\Oc_m$, and so does $a$.
Now the assertion follows from the argument of Lemma \ref{lem:RSSP}.
\end{proof}
For $\vp \in \Ss(M_n)$, let $\vp^\sharp$ be the Fourier transform of $\vp$: 
\begin{align*}
\vp^\sharp(x) = \int_{G_n} \vp(y)\psi(\tr(yx)) \dr y
\end{align*} 
where $\dr y$ indicates the self-dual Haar measure on $M_n$.
We define
\begin{align*}
\vp_{\cf_\pi}^*(x) = \vp_{\cf_\pi}^{\sharp}(\upsilon_{\cf_\pi}^{-1} {}^t x w_n), 
\end{align*} 
where $\upsilon_{\cf_\pi}$ is the matrix (\ref{def:ups}).
Then, $J_\pi^*\vp_{\cf_\pi}^*$ is $\K_{\cf_\pi}^*$-invariant, and
\begin{align}
q^{\cf_\pi s}Z(s,J_\pi^*,\vp_{\cf_\pi}^*) = Z(s,J_\pi^\vee, \vp_{\cf_\pi}^\sharp), \label{eq:vpcf}
\end{align}  
where $J^\vee$ is defined by $J^\vee(g) = J(g^{-1})$.
We give an explicit description for $\vp_{\cf_\pi}^*$.
In the case where $\chi$ is unramified, it is given by
\begin{align*}
\vp_{\cf_\pi}^*(\begin{bmatrix}
a & b  \\
c& d
\end{bmatrix}) = \Ch(\begin{bmatrix}
a & b \\
c& d
\end{bmatrix};R_{\cf_\pi}^*)\prod_{j=2}^{m}\Ch(b_{j1}; \p^{l}),
\end{align*} 
where $R_{\cf_\pi}^*$ is the ring (c.f. \ref{eq:Rc*}).
In the case where $\chi$ is ramified, define $\phi_\chi^{\circ *} \in \Ss(M_{m-1})$ by
\begin{align*}
\phi_\chi^{\circ*}(x) = \prod_{1 \le i \neq j \le m-1}\Ch(x_{ij};\oo) \prod_{i=1}^{m-1} \chi_0^{-1}(x_{ii}).
\end{align*} 
Define $\phi_{\chi,l}^* \in \Ss(M_{m})$ by
\begin{align*}
\phi_{\chi,l}^*(\begin{bmatrix}
w &  y \\
{}^tz & x
\end{bmatrix}) = \frac{\Ch(y, z; \oo^{m-1} \tm (\p^{l})^{m-1}) \chi_0^{-1}(w)}{\vol(\SL_{m-1}(\oo))} \int_{\SL_{m-1}(\oo)}\phi_\chi^{\circ*}(x  u) \dr u. 
\end{align*} 
Then, the explicit form of $\vp_{\cf_\pi}^*$ is 
\begin{align*}
\vp_{\cf_\pi}^*(\begin{bmatrix}
a & b  \\
c& d
\end{bmatrix}) = \gf(\chi,\psi_{\Vp^{-e}})^{m}\phi_{\chi,l}^*(d) \Ch(\begin{bmatrix}
a & b \\
c& 
\end{bmatrix};R_{\cf_\pi}^*)\prod_{j=2}^{m}\Ch(b_{j1}; \p^{-l}).
\end{align*} 
For $k \in \Z$, let 
\begin{align*}
B_{m,k}^e = \{b \in B_m \mid o(\det(b)) = k, b_{11} \in \oo, b_{1j} \in \p^{e-l}; j >1 , b_{ij} \in \oo; i>1 \},
\end{align*} 
and 
\begin{align*}
c_k = \sum_{B_{m,k}^e/B_{m,0}^e} \ac{J}_\pi^*(b).
\end{align*} 
Similar to Prop \ref{prop:ZJPhi}, we can show:
\begin{prop}\label{prop:ZJ*Phi*}
With the above notation, 
\begin{align*}
Z(s + \frac{n-1}{2}, J_\pi^*, \vp_{\cf_\pi}^*) = q^{l(m-1)}\gf(\chi,\psi_{\Vp^{-e}})^{m} \vol(\K(\cf_\pi)^*)\sum_{i = 0}^\infty c_i^* q^{i(-s+\frac{1}{2})}.
\end{align*} 
\end{prop}
From now, assume that $\pi$ is unitary.
Since $\pi$ is unitary, its Shalika forms are matrix coefficients.
From (\ref{eq:vpcf}), and the Godemant-Jacquet functional equation (\cite{G-J}), it follows that 
\begin{align*}
\frac{Z(1-s+ \frac{n-1}{2}, J_\pi^*, \vp_{\cf_\pi})}{L(1-s,\pi^\vee)} = \ep_\pi \frac{Z(s+\frac{n-1}{2},J_\pi, \vp_{\cf_\pi})}{L(s,\pi)} 
\end{align*}
where both the $L$-factors and the root number $\ep_\pi$ are same as those defined by Whittaker forms (c.f. \cite{J-PS-S2}).
By Prop. \ref{prop:ZJPhi}, \ref{prop:ZJ*Phi*}, the right and left hand side lie in $\C[q^{-s}]$ and $\C[q^{s}]$, respectively.
Hence, both sides are nonzero constant.
Thus we have:
\begin{thm}\label{thm:pairnf}
With the above notation and assumptions,
\begin{align*}
\sum_{i = 0}^\infty c_i q^{i(-s+\frac{1}{2})} = L(s,\pi), \
\sum_{i = 0}^\infty c_i^* q^{i(-s+\frac{1}{2})} = \ep_\pi L(s,\pi^\vee).
\end{align*} 
\end{thm}

\end{document}